\documentclass{amsart}
\usepackage{graphicx}
\usepackage[dvipsnames]{xcolor}
\usepackage{verbatim, bm, hyperref}
\hypersetup{linkbordercolor=Blue, citebordercolor=Blue, urlbordercolor=Blue}

\usepackage{amsthm, amsmath, amssymb, parskip, tikz-cd, relsize, amsthm}
\usepackage[nobysame]{amsrefs}
\usepackage[cal=rsfso, calscaled=1.03]{mathalpha}

\graphicspath{ {/Users/nadiya.upegui.keagy/Documents/PhD Program/LaTeX Files/LaTeX Images/} }

\newcommand{\defn}[1]{\emph{#1}}
\newcommand{\RMod}{R\text{-}\textbf{Mod}}

\newcommand{\Z}{\mathbb{Z}}
\newcommand{\N}{\mathbb{N}}
\newcommand{\Ker}{\text{ker}}
\newcommand{\Coker}{\text{coker}}
\newcommand{\Colim}{\text{colim}}
\newcommand{\Ima}{\text{im}}
\newcommand{\Id}{\text{id}}
\newcommand{\Set}[2]{\{#1 : #2\}}

\newcommand{\Frac}[2]{\mathlarger{\frac{#1}{#2}}}
\newcommand{\Inv}[1]{#1^{-1}}

\newcommand{\End}{\text{End}}
\newcommand{\Vect}{\text{Vec}}

\theoremstyle{plain}
\newtheorem{Theorem}{Theorem}[section]
\newtheorem{Lemma}[Theorem]{Lemma}
\newtheorem{Corollary}[Theorem]{Corollary}
\newtheorem{Proposition}[Theorem]{Proposition}
\theoremstyle{definition}
\newtheorem{Definition}[Theorem]{Definition}
\newtheorem{Definition-Lemma}[Theorem]{Definition-Lemma}

\newtheorem{Notation}[Theorem]{Notation}

\newtheorem{maintheorem}{Theorem}

\newtheorem{maincorollary}[maintheorem]{Corollary}

\DefineSimpleKey{bib}{primaryclass}{}
\DefineSimpleKey{bib}{archiveprefix}{}

\BibSpec{arXiv}{%
  +{}{\PrintAuthors}{author}
  +{,}{ \textit}{title}
  +{}{ \parenthesize}{date}
  +{,}{ arXiv }{eprint}
  +{,}{ primary class }{primaryclass}
  +{,}{ note }{note}
}

\title[Projective Constant Decompositions of Persistence Modules]{Projective Constant Decompositions of Persistence Modules over Noetherian Rings}
\date{}
\author{Nadiya Upegui Keagy}

\begin{document}

\begin{abstract}
   Persistence modules serve as the algebraic foundation for topological data analysis, typically studied as representations of posets over a field. This article extends the structural and decomposition theory of persistence modules to the more general setting of unitary left modules over Noetherian rings. We introduce the notion of a module of projective constant type, a generalization of interval modules that facilitates decomposition results. We characterize the existence of projective constant decompositions for pointwise finitely generated persistence modules indexed by $A_n$-type quivers, totally ordered sets, and zigzag posets. Our primary results establish that such decompositions exist if and only if specific algebraic criteria are met: the projective colimit conditions (PCC) for quiver and zigzag indexings, and the projectivity of cokernels of internal morphisms for totally ordered indexings. Furthermore, we provide corollaries for the existence of classical interval decompositions over Noetherian rings where finitely generated projective modules are free, such as principal ideal domains (PIDs). Finally, we establish the uniqueness of indecomposable projective constant decompositions within the Krull-Schmidt framework and extend the uniqueness of interval decompositions to persistence modules over integral domains.
\end{abstract}

\maketitle

\section*{Introduction}\label{sec:intro}

   Persistence modules arose in topological data analysis as an algebraic framework for encoding how topological features of a space evolve across a filtration. Initially motivated by problems in applied topology and data analysis, a persistence module assigns a vector space to each parameter value and linear maps between them that reflect inclusion-induced structure. Among the most fundamental examples are interval modules, which serve as the basic building blocks in decomposition results and underlie the barcode representation of persistence. 

   Originally defined as finite sequences of finite-dimensional vector spaces connected by linear maps, persistence modules were later formalized by Zomorodian and Carlsson as quiver representations indexed by the natural numbers and linked to graded modules over polynomial rings \cite{zom05}. Cohen-Steiner, Edelsbrunner, and Harer extended this framework to real-indexed persistence modules, initiating the study of continuous persistence \cite{coh07}. Subsequent work by Carlsson and de Silva introduced zigzag modules and related them to finite-dimensional representations of $A_n$-type quivers, thereby bringing Gabriel's theorem and tools from quiver representation theory into the persistence setting \cite{car10}.

   More recently, Botnan and Crawley-Boevey established structural and interval decomposition theorems for pointwise finite-dimensional persistence modules indexed by totally ordered sets and zigzag paths \cite{bot18}. Furthermore, Botnan demonstrated the existence of interval decompositions for persistence modules indexed by infinite discrete zigzag posets \cite{bot17}. We have also contributed alternative constructive proofs for the existence of interval decompositions of persistence modules indexed by both totally ordered sets and zigzag posets \cite{gan25}.

   Although persistence modules have predominantly been studied over fields, recent work has extended these investigations to persistence modules over more general coefficient rings. Luo and Henselman-Petrusek established criteria for the existence of interval decompositions of pointwise freely and finitely generated persistence modules over principal ideal domains (PIDs), initially for modules indexed by finite linear quivers \cite{luo23}, and subsequently for modules indexed by arbitrary totally ordered sets \cite{luo25}.

   The purpose of this article is to extend several structural results for persistence modules to the setting of unitary left modules over Noetherian rings $R$ with identity (note that we omit the qualifier ``unitary left'' as no confusion will arise). To do this, we introduce a generalized notion of an interval module, called a \emph{module of projective constant type} (see Definition~\ref{def:pc-module}). This module behaves in essentially the same way in that it is isomorphic to some fixed projective $R$-module on indices lying in the interval and zero elsewhere, with isomorphisms between nonzero components. Some motivation for this definition stems from the fact that an interval module is itself a module of projective constant type.

   Fundamentally, we wish to understand when a persistence module can be decomposed as an internal direct sum of modules of projective constant type, which we call a \emph{projective constant decomposition} (see Definition~\ref{def:pc-decomp}). In this article, we prove that such a decomposition exists for pointwise finitely generated persistence modules indexed by an $A_n$-type quiver or a zigzag poset if and only if particular conditions are satisfied, which we call the \emph{projective colimit conditions (PCC)} (see Definition~\ref{def:proj-colim-cond}). We also prove that a pointwise finitely generated persistence module indexed by a totally ordered poset admits a projective constant decomposition if and only if the cokernel of every internal morphism is projective. Our main contributions are stated in the proceeding theorems and corollaries. 

   \begin{maintheorem}\label{thm:main-pc-quiver}
      Let $R$ be a (not necessarily commutative) Noetherian ring with identity. Let $F:A_n \to \RMod$ be a nonzero pointwise finitely generated $A_n$-persistence module. Suppose that $F$ satisfies the projective colimit conditions (PCC). Then $F$ admits a projective constant decomposition.
   \end{maintheorem}

   \begin{maintheorem}\label{thm:main-pc-total}
      Let $R$ be a commutative Noetherian ring with identity. Let $F:P \to \RMod$ be a nonzero pointwise finitely generated totally ordered persistence module. Suppose that the cokernel of every internal morphism $F_{x,y}$ is projective. Then $F$ admits a projective constant decomposition.
   \end{maintheorem}

   \begin{maintheorem}\label{thm:main-pc-zigzag}
      Let $R$ be a commutative Noetherian ring with identity. Let $F:P \to \RMod$ be a nonzero pointwise finitely generated zigzag persistence module with extrema $\{z_i \mid i \in \Gamma \}$. Suppose that $F$ satisfies the projective colimit conditions (PCC). Then $F$ admits a projective constant decomposition.
   \end{maintheorem}

   The proofs of Theorem~\ref{thm:main-pc-quiver} and of Theorem~\ref{thm:main-pc-zigzag} for zigzag posets with finite extrema are inspired by Ringel’s decomposition of representations of $A_n$-type quivers into thin representations~\cite{rin16}. The proof of Theorem~\ref{thm:main-pc-total} generalizes our earlier argument showing the existence of interval decompositions for persistence modules over fields indexed by totally ordered sets~\cite{gan25}*{\S6}. Finally, the proof of Theorem~\ref{thm:main-pc-zigzag} in the case of zigzag posets with infinite extrema adapts the approach introduced by Botnan for decomposing persistence modules over fields indexed by infinite discrete zigzag posets~\cite{bot17}.

   Consequently, we also contribute the following corollaries regarding the existence of interval decompositions when, in addition to being a (commutative) Noetherian ring with identity, $R$ has the property that every finitely generated projective $R$-module is free. Prominent examples of such rings include principal ideal domains, polynomial rings in finite variables over a PID or field, local Noetherian rings with identity, and rings of formal power series over a local Noetherian ring with identity. We discuss these examples further in Corollary~\ref{cor:free-interval-decomp-examples}.

   \begin{maincorollary}\label{cor:main-pid-quiver}
      Let $R$ be a (not necessarily commutative) Noetherian ring with identity such that every finitely generated projective $R$-module is free. Let $F:A_n \to \RMod$ be a nonzero pointwise finitely generated $A_n$-persistence module. Suppose that $F$ satisfies the projective colimit conditions (PCC). Then $F$ admits an interval decomposition.
   \end{maincorollary}

   \begin{maincorollary}\label{cor:main-pid-total}
      Let $R$ be a commutative Noetherian ring with identity such that every finitely generated projective $R$-module is free. Let $F:P \to \RMod$ be a nonzero pointwise finitely generated totally ordered persistence module. Suppose that the cokernel of every internal morphism $F_{x,y}$ is projective. Then $F$ admits an interval decomposition.
   \end{maincorollary}

   \begin{maincorollary}\label{cor:main-pid-zigzag}
      Let $R$ be a commutative Noetherian ring with identity such that every finitely generated projective $R$-module is free. Let $F:P \to \RMod$ be a nonzero pointwise finitely generated zigzag persistence module with extrema $\{z_i \mid i \in \Gamma \}$. Suppose that $F$ satisfies the projective colimit conditions (PCC). Then $F$ admits an interval decomposition.
   \end{maincorollary}

   In the specific case when $R$ is a PID, Corollary~\ref{cor:main-pid-total} was proven by Luo and Henselman-Petrusek~\cite{luo25}. 

   We contribute the following theorems to confirm that the stated conditions are indeed necessary whenever such projective constant (respectively, interval) decompositions exist.

   \begin{maintheorem}\label{thm:main-pc-quiver-nec}
      Let $R$ be a (not necessarily commutative) Noetherian ring with identity. Let $F:A_n \to \RMod$ be a nonzero pointwise finitely generated $A_n$-persistence module. Suppose that $F$ admits a projective constant (respectively, interval) decomposition. Then $F$ satisfies the projective colimit conditions (PCC). 
   \end{maintheorem}

   \begin{maintheorem}\label{thm:main-pc-total-nec}
      Let $R$ be a (not necessarily commutative) Noetherian ring with identity. Let $F:P \to \RMod$ be a nonzero pointwise finitely generated totally ordered persistence module. Suppose that $F$ admits a projective constant (respectively, interval) decomposition. Then the cokernel of every internal morphism $F_{x,y}$ is projective.
   \end{maintheorem}

   \begin{maintheorem}\label{thm:main-pc-zigzag-nec}
      Let $R$ be a (not necessarily commutative) Noetherian ring with identity. Let $F:P \to \RMod$ be a nonzero pointwise finitely generated zigzag persistence module with extrema $\{z_i \mid i \in \Gamma \}$. Suppose that $F$ admits a projective constant (respectively, interval) decomposition. Then $F$ satisfies the projective colimit conditions (PCC). 
   \end{maintheorem}

   In the specific context of interval decompositions over a PID, Theorem~\ref{thm:main-pc-total-nec} was proven by Luo and Henselman-Petrusek~\cite{luo25}. 

   Finally, we contribute the following uniqueness results for indecomposable projective constant decompositions in the setting of Krull-Schmidt categories and for interval decompositions over integral domains.

   \begin{maintheorem}\label{thm:main-krull-uniqueness}
      Let $R$ be a (not necessarily commutative) ring with identity. Let $P$ be a poset and let $F:P \to \RMod$ be a nonzero persistence module. Suppose that $\RMod$ is a Krull-Schmidt category. If $F$ admits a projective constant decomposition, then $F$ admits a unique indecomposable projective constant decomposition up to isomorphism and permutation of the summands.
   \end{maintheorem}

   \begin{maintheorem}\label{thm:main-domain-uniqueness}
      Let $R$ be an integral domain. Let $P$ be a small poset category and let $F:P \to \RMod$ be a nonzero pointwise finitely generated persistence module. If $F$ admits an interval decomposition, then this decomposition is unique up to isomorphism and permutation of the summands.
   \end{maintheorem}

   The proof of Theorem~\ref{thm:main-krull-uniqueness} leverages the renowned Krull-Remak-Schmidt-Azumaya theorem~\cite{azu50}, while the proof of Theorem~\ref{thm:main-domain-uniqueness} is inspired by Luo and Henselman-Petrusek's proof of the uniqueness of interval decompositions of pointwise freely and finitely generated persistence modules over PIDs indexed by finite linear quivers \cite{luo23}*{\S4, Lemma 1}.

   This article is organized as follows: Sections~\ref{sec:prelims} and \ref{sec:decomp-criteria} provide background and technical lemmas regarding persistence modules over Noetherian rings that will be useful in proving our main results. In Section~\ref{sec:zigzag-posets}, we define zigzag posets and zigzag persistence modules, and discuss some of their properties. Theorem~\ref{thm:main-pc-quiver} and Corollary~\ref{cor:main-pid-quiver} are proved in Section~\ref{sec:quiver-zigzag} on $A_n$-persistence modules; see Theorem~\ref{thm:quiver-pc-decomp} and Corollary~\ref{cor:quiver-pid-interval-decomp}. Theorem~\ref{thm:main-pc-total} and Corollary~\ref{cor:main-pid-total} are proved in Section~\ref{sec:total-order} on totally ordered persistence modules; see Theorem~\ref{thm:total-complete-pc-decomp} and Corollary~\ref{cor:total-pid-decomp}. Theorem~\ref{thm:main-pc-zigzag} and Corollary~\ref{cor:main-pid-zigzag} are proved in Section~\ref{sec:finite-zigzag} for zigzag persistence modules with finite extrema (see Theorem~\ref{thm:fin-zigzag-pc-decomp} and Corollary~\ref{cor:fin-zigzag-int-decomp}) and in Section~\ref{sec:infinite-zigzag} for zigzag persistence modules with infinite extrema (see Theorem~\ref{thm:inf-complete-pc-decomp}, Theorem~\ref{thm:inf-zigzag-pc-decomp-nat}, and Corollary~\ref{cor:inf-pid-zigzag-pc-decomp}). Theorems~\ref{thm:main-pc-quiver-nec}, \ref{thm:main-pc-total-nec}, and \ref{thm:main-pc-zigzag-nec}, which establish the necessity of the given conditions to obtain decomposition results, are all proven in Section~\ref{sec:necessity}; see Theorems~\ref{thm:nec-zigzag-pc},~\ref{thm:nec-quiver-pc}, and \ref{thm:nec-total-pc}. Theorems~\ref{thm:main-krull-uniqueness} and \ref{thm:main-domain-uniqueness} regarding uniqueness are proved in Section~\ref{sec:uniqueness}; see Theorems~\ref{thm:krull-uniqueness} and ~\ref{thm:domain-uniqueness}. 

   This article is essentially self-contained; we use only elementary results in module theory and Zorn's Lemma.

\section{Definitions and Notation}\label{sec:prelims}

   Fix a Noetherian ring $R$ with identity, and let $\RMod$ denote the category of unitary left $R$-modules. Throughout this article, we omit the qualifier ``unitary left'' when referring to such $R$-modules, as no ambiguity will arise from context.

   %%% Definition: Persistence Module %%%
   \begin{Definition}
      A $P$-\defn{persistence module} (or just \defn{persistence module}) is a functor $F:P \to \RMod$ for some fixed poset category $P$. A \defn{morphism} of persistence modules indexed by $P$ is, by definition, a natural transformation of functors. The direct sum of persistence modules, persistence submodule, and restriction of a persistence module to a subposet of $P$ are all defined in the natural way. 
   \end{Definition}

   %%% Notation: Basic Persistence Module Notation %%%
   \begin{Notation}
      Let $\leq$ denote the partial ordering of the poset $P$. For $x \leq y$ in $P$, denote the $R$-module $F(x)$ by $F_x$ and the \defn{internal morphisms} $F(x \to y)$ by $F_{x,y}$. Since $F$ is a functor, we have $F_{x,x} = \Id_{F_x}$ and also $F_{y,z} \circ F_{x,y} = F_{x,z}$ for any $x \leq y \leq z$. Let $\Ker[x,y]=\Ker(F_{x,y})$, $\Ima[x,y]=\Ima(F_{x,y})$ and $\Coker[x,y]= \Coker(F_{x,y})$ with regard to the internal morphisms.
   \end{Notation}

   %%% Definition: Pointwise Finitely-Generated %%%
   \begin{Definition}\label{def:pfg}
      Let $P$ be a poset and let $F:P \to \RMod$ be a persistence module. We say $F$ is \defn{pointwise finitely generated (p.f.g.)} if $F_x$ is a finitely generated $R$-module for every $x \in P$. 
   \end{Definition}

   %%% Defintion: Interval of a Poset %%%
   \begin{Definition}\label{def:interval}
      Let $P$ be a poset and $I$ a subset of $P$. 
      \begin{enumerate}
         \item We say that $I$ is \defn{convex} if for all $x,y,z\in P$, we have $y\in I$ whenever $x,z\in I$ and $x \leq y \leq z$. 
         \item We say that $I$ is \defn{connected} if for all $x,z\in I$, there exist $y_0, y_1, \ldots, y_n\in I$ such that
         \begin{center}
         $\begin{cases}
            y_0=x, \\
            y_n=z, \\
            y_{i-1}\leq y_i \mbox{ or } y_{i-1} \geq y_i \mbox{ for each }i\in \{1, \ldots, n\}.
         \end{cases}$
         \end{center}
      \item We call $I$ an \defn{interval} in $P$ if $I$ is nonempty, convex, and connected.  
      \end{enumerate}
   \end{Definition}

   %%% Definition: F(I,M) %%%
   \begin{Definition}
      Let $P$ be a poset and $I$ an interval in $P$. Given an $R$-module $M$, define $F(I,M)$ by:
      \begin{enumerate}
      \item for all $x\in P$,
      \[ F(I,M)_x = \begin{cases}
         M & \mbox{ if }x \in I,\\
         0 & \mbox{ if }x\notin I;
         \end{cases} \] 
      \item for all $x \leq y$ in $P$,
      \[ F(I,M)_{x,y} = \begin{cases}
         \mathrm{id}_M & \mbox{ if }x,y\in I,\\
         0 & \mbox{ if }x\notin I \mbox{ or }y\notin I.
         \end{cases} \]
      \end{enumerate}
   \end{Definition}

   %%% Definition: Projective Constant Type %%%
   \begin{Definition}\label{def:pc-module}
      Let $P$ be a poset. We say a persistence module indexed by $P$ is \defn{of projective constant type} if it is isomorphic to $F(I,M)$ for some interval $I$ in $P$ and for some projective $R$-module $M$.
   \end{Definition}

   %%% Definition: Projective Constant Decomposition %%%
   \begin{Definition}\label{def:pc-decomp}
      Let $P$ be a poset and let $F:P \to \RMod$ be a persistence module. A \defn{projective constant decomposition} of $F$ is an internal direct sum decomposition $F = \bigoplus_{G \in \mathcal{A}} G$ where $\mathcal{A}$ is a set of persistence submodules of $F$ such that every $G \in \mathcal{A}$ is of projective constant type. 
   \end{Definition}

   %%% Definition: Interval Module %%%
   \begin{Definition}
      Let $P$ be a poset. An \defn{interval module} is a persistence module indexed by $P$ which is isomorphic to $F(I,R)$ for some interval $I$ in $P$. 
   \end{Definition}

   %%% Definition: Interval Decomposable %%%
   \begin{Definition}
      Let $P$ be a poset and let $F:P \to \RMod$ be a persistence module. An \defn{interval decomposition} of $F$ is an internal direct sum decomposition $F = \bigoplus_{G \in \mathcal{A}} G$ where $\mathcal{A}$ is a set of persistence submodules of $F$ such that every $G \in \mathcal{A}$ is an interval module. If $F$ admits an interval decomposition, we say $F$ is \defn{interval decomposable}. 
   \end{Definition}

   Note that an interval decomposition is a special case of a projective constant decomposition since $R$ is free and hence projective. Conversely, if every finitely generated projective $R$-module is free, we obtain the following useful lemma.

   %%% Lemma: Interval Decomposition %%%
   \begin{Lemma}\label{lem:free-interval-decomp}
      Let $R$ be a Noetherian ring with identity such that every finitely generated projective $R$-module is free. Let $P$ be a poset and let $F:P \to \RMod$ be a p.f.g. persistence module. Suppose that $F$ admits a projective constant decomposition. Then $F$ is interval decomposable.
   \end{Lemma}

   \begin{proof}
      Let $F=\bigoplus_{G \in \mathcal{A}} G$ be a projective constant decomposition of $F$. Then each $G \in \mathcal{A}$ is isomorphic to $F(I,M)$ for some interval $I$ in $P$ and for some finitely generated projective $R$-module $M$. By assumption, $M$ is free, so $G$ is isomorphic to a direct sum of copies of $F(I,R)$. Thus, each $G$ admits an interval decomposition, which shows that $F$ is interval decomposable.
   \end{proof}

   The following corollary presents some notable examples of rings that satisfy the conditions of Lemma~\ref{lem:free-interval-decomp}.

   %%% Corollary: Examples of Free Module Rings %%%
   \begin{Corollary}\label{cor:free-interval-decomp-examples}
      Let $R$ be one of the following rings:
      \begin{enumerate}
         \item a principal ideal domain;
         \item the polynomial ring in finite variables $k[x_1, \ldots, x_n]$, where $k$ is a principal ideal domain or a field;
         \item a local (commutative) Noetherian ring with identity;
         \item the ring of formal power series $S[[x]]$, where $S$ is a local (commutative) Noetherian ring with identity.
      \end{enumerate}
      Let $P$ be a poset and let $F:P \to \RMod$ be a p.f.g. persistence module. Suppose that $F$ admits a projective constant decomposition. Then $F$ is interval decomposable.
   \end{Corollary}

   \begin{proof}
      Immediate from Lemma \ref{lem:free-interval-decomp} by the following observations: 
      \begin{enumerate}
      \item Every projective $R$-module over a PID is free. 
      
      \item Quillen and Suslin independently proved that when $k$ is a PID or a field, every finitely generated projective module over the polynomial ring \\ $k[x_1, \ldots, x_n]$ is free; this result is now known as the Quillen–Suslin Theorem \cites{qui76, sus76}. Moreover, Hilbert's Basis Theorem ensures that $k[x_1, \ldots, x_n]$ is a Noetherian ring with identity when $k$ is a PID or a field \cite{hil90}.
      
      \item Kaplansky proved that every projective $R$-module over a (not necessarily commutative) local ring with identity is free \cite{kap58}.  
      
      \item When $S$ is a local (commutative) Noetherian ring with identity, $S[[x]]$ is also a local (commutative) Noetherian ring with identity. Case (3) implies that every projective $S[[x]]$-module is free.
      \end{enumerate}
   \end{proof}

   %%% Definition: Source and Sink %%%
   \begin{Definition}\label{def:source-and-sink}
      Let $P$ be a poset and $x \in P$. We call an index $x$ a \defn{sink} if there are no morphisms $x \to y$ for $x \neq y$. Dually, we call an index $x$ a \defn{source} if there are no morphisms $y \to x$ for $x \neq y$. 
   \end{Definition}

\section{Decomposition Criteria over Noetherian Rings}\label{sec:decomp-criteria}

   Recall that a finitely generated projective module over a Noetherian ring $R$ with identity is itself Noetherian and therefore satisfies the ascending chain condition. This property reflects the primary motivation for working over Noetherian rings, which imposes the necessary finiteness condition to obtain decomposition results. In this section, we establish two auxiliary results with regard to Noetherian $R$-modules. The first shows that certain chains of submodules must be finite, while the second provides a general criterion for decomposing a persistence module as a direct sum of submodules drawn from a specified collection.

   %%% Lemma: Finite Chain Lemma %%%
   \begin{Lemma}\label{lem:fin-chain}
      Let $R$ be a Noetherian ring with identity. Let $M$ be a finitely generated projective $R$-module, and let $S$ be a chain of $R$-submodules of $M$. Assume for all $A \in S$ that $M/A$ is projective. Then $S$ is a finite set.
   \end{Lemma}

   \begin{proof}
      Since $M$ is Noetherian, the set $S$ has a maximal element, say $A_1$. Now $M/A_1$ projective implies $M=N_1 \oplus A_1$ for some projective $R$-submodule $N_1$ of $M$. 

      If $S=\{A_1\}$, we are done. If not, the set $S \setminus \{A_1\}$ has a maximal element, say $A_2$. We have a short exact sequence

      %%% S.E.S. %%%
      \begin{center}
         \begin{tikzcd}
            0 \arrow [r] 
            & {A_1}/{A_2} \arrow[r, hook]
            & {M}/{A_2} \arrow[r, two heads]
            & {M}/{A_1} \arrow[r] 
            & 0.
         \end{tikzcd}
      \end{center}

      The sequence splits since $M/A_1$ is projective, and thus $A_1/A_2$ is projective as it is a direct summand of $M/A_2$ which is projective. In particular, there exists a nonzero $R$-submodule $N_2$ of $A_1$ such that $A_1=N_2 \oplus A_2$. 
      
      If $S=\{A_1,A_2\}$, we are done. If not, repeat this process for $A_3$ and so on. This process must terminate after a finite number of steps, for otherwise 
      \begin{center}
         $N_1 \subset N_1 \oplus N_2 \subset N_1 \oplus N_2 \oplus N_3 \subset \cdots$
      \end{center}
      is an infinite ascending chain of $R$-submodules of $M$, a contradiction since $M$ satisfies the ascending chain condition.
   \end{proof}

   The proof of the following lemma is essentially the proof in \cite{gan25}*{Lemma 3.6} with the minor addition of Claim~\ref*{lem:decomp-criteria}.1.

   %%% Lemma: Decomposition Criteria for Noetherian Rings%%%
   \begin{Lemma} \label{lem:decomp-criteria}
      Let $R$ be a Noetherian ring with identity. Let $P$ be a poset and let $F:P \to \RMod$ be a p.f.g. persistence module. Let $\mathcal{E}$ be a set of nonzero submodules of $F$. Assume that for each nonzero direct summand $H$ of $F$, there exists a direct summand $G$ of $H$ such that $G \in \mathcal{E}$. Then there exists a subset $\mathcal{A}$ of $\mathcal{E}$ giving an internal direct sum decomposition $F=\bigoplus_{G \in \mathcal{A}} G$. 
   \end{Lemma}

   \begin{proof}
      Let $\mathbf{S}$ be the set of all pairs $(\mathcal{A}, H)$ where: 
      \begin{quote}
         $\mathcal{A}$ is a subset of $\mathcal{E}$, 

         $H$ is a submodule of $F$,
      \end{quote}
      such that we have an internal direct sum decomposition $F=\left(\bigoplus_{G \in \mathcal{A}} G \right) \oplus H$. The set $\mathbf{S}$ is nonempty because $(\emptyset, F)\in \mathbf{S}$.

      Define a partial order on $\mathbf{S}$ by $(\mathcal{A}, H) \leq (\mathcal{B}, J)$ if $\mathcal{A} \subseteq \mathcal{B}$ and
      \[ H = \left( \bigoplus_{G \in \mathcal{B} - \mathcal{A}} G  \right) \oplus J. \] 

      %%% Claim 1: Minimal J_x in Chain %%%
      \textbf{Claim~\ref*{lem:decomp-criteria}.1.} For any chain $\mathbf{T}$ in $\mathbf{S}$ and fixed $x \in P$, there exists a pair $(\mathcal{B}, J)$ in $\mathbf{T}$ such that for all $(\mathcal{A},H)$ in $\mathbf{T}$, we have $J_x \subseteq H_x$.

      \begin{proof}[Proof of Claim~\ref*{lem:decomp-criteria}.1]
         Suppose by way of contradiction that no such pair $(\mathcal{B}, J)$ exists. Choose any $(\mathcal{A}^1, H^1)$ in $\mathbf{T}$. Then there exists $(\mathcal{A}^2, H^2)$ in $\mathbf{T}$ such that $H^2_x \subsetneq H^1_x$. Since $\mathbf{T}$ is a chain, we have either $(\mathcal{A}^1, H^1) \leq (\mathcal{A}^2, H^2)$ or $(\mathcal{A}^2, H^2) \leq (\mathcal{A}^1, H^1)$. The latter case is not possible since this would imply $H^1 \subseteq H^2$. Therefore $(\mathcal{A}^1, H^1) \leq (\mathcal{A}^2, H^2)$. 

         Repeating this process, we obtain an infinite sequence 
         \begin{center}
            $(\mathcal{A}^1, H^1), (\mathcal{A}^2, H^2), (\mathcal{A}^3, H^3), \ldots$
         \end{center}
         in $\mathbf{T}$. But then
         \begin{center}
            $\bigoplus\limits_{G \in \mathcal{A}^1} G_x \subset \bigoplus\limits_{G \in \mathcal{A}^2} G_x \subset \bigoplus\limits_{G \in \mathcal{A}^3} G_x \subset \cdots$
         \end{center}
         is an infinite ascending chain of $R$-submodules of $F_x$, a contradiction since $F_x$ is Noetherian and satisfies the ascending chain condition. 
      \end{proof} 

      %%% Claim 2: Every Chain has an Upper Bound %%%
      \textbf{Claim~\ref*{lem:decomp-criteria}.2.} Every chain in $\mathbf{S}$ has an upper bound. More precisely:
      \begin{quote} 
         For any chain $\mathbf{T}$ in $\mathbf{S}$, let 
         \begin{align*} 
            \mathcal{C} = \bigcup_{(\mathcal{A}, H)\in \mathbf{T}} \mathcal{A}, \qquad
            N = \bigcap_{(\mathcal{A}, H)\in \mathbf{T}} H.
         \end{align*}
         Then $(\mathcal{C}, N)$ lies in $\mathbf{S}$ and is an upper bound of $\mathbf{T}$.
      \end{quote}

      \begin{proof}[Proof of Claim~\ref*{lem:decomp-criteria}.2]    
         There are two things to check: 
         \begin{enumerate}
            \item $F= \left(\bigoplus_{G\in \mathcal{C}} G\right) \oplus N$; 
         
            \item $H=\left( \bigoplus_{G\in \mathcal{C}-\mathcal{A}} G\right) \oplus N$ for all $(\mathcal{A}, H)\in \mathbf{T}$.
         \end{enumerate}

         Thus for each $x \in P$, we need to check:
         \begin{enumerate}
            \item $F_x= \left(\bigoplus_{G \in \mathcal{C}} G_x\right) \oplus N_x$; 
         
            \item $H_x=\left( \bigoplus_{G \in \mathcal{C}-\mathcal{A}} G_x\right) \oplus N_x$ for all $(\mathcal{A}, H)\in \mathbf{T}$.
         \end{enumerate}

         Fix $x \in P$. By Claim~\ref*{lem:decomp-criteria}.1, we can choose a pair $(\mathcal{B}, J) \in \mathbf{T}$ such that for all $(\mathcal{A},H)$ in $\mathbf{T}$, we have $J_x \subseteq H_x$. Now let $(\mathcal{A}, H)\in \mathbf{T}$. We make the following observations:
         \begin{enumerate}
            \item[(a)] If $(\mathcal{A}, H) \leq (\mathcal{B}, J)$, then 
            \[ H_x = \left( \bigoplus_{G \in \mathcal{B} - \mathcal{A}} G_x  \right) \oplus J_x; \] 
            hence $J_x \subseteq H_x$. 
            
            \item[(b)] If $(\mathcal{B}, J) \leq (\mathcal{A}, H)$, then 
            \[ J_x = \left( \bigoplus_{G \in \mathcal{A} - \mathcal{B}} G_x  \right) \oplus H_x. \] 
            
            The above equality shows $H_x \subseteq J_x$. But also $J_x \subseteq H_x$, hence $J_x = H_x$ and $G_x = 0$ for all $G \in \mathcal{A}-\mathcal{B}$. 

            \item[(c)] By (a) and (b), we have $N_x = J_x$, and $G_x=0$ for all $G \in \mathcal{C}-\mathcal{B}$. 
         \end{enumerate}

         We deduce that:
         \begin{align*}
            F_x &= \left( \bigoplus_{G \in \mathcal{B}} G_x \right) \oplus J_x & \\
            &= \left( \bigoplus_{G \in \mathcal{C}} G_x \right) \oplus N_x & \mbox{by (c).}
         \end{align*}
         If $(\mathcal{A}, H) \leq (\mathcal{B}, J)$, then
         \begin{align*}
            H_x &= \left( \bigoplus_{G \in \mathcal{B} - \mathcal{A}} G_x  \right) \oplus J_x & \\
            &= \left( \bigoplus_{G \in \mathcal{C} - \mathcal{A}} G_x  \right) \oplus N_x & \mbox{by (c).}
         \end{align*}
         If $(\mathcal{B}, J) \leq (\mathcal{A}, H)$, then
         \begin{align*}
            H_x &= J_x & \mbox{by (b)} \\
            &=  \left( \bigoplus_{G \in \mathcal{C} - \mathcal{A}} G_x  \right) \oplus N_x & \mbox{by (c).}
         \end{align*}
      \end{proof}

      By Claim~\ref*{lem:decomp-criteria}.2, we can apply Zorn's lemma to deduce there exists a pair, say $(\mathcal{A}, H)$, which is maximal in $\mathbf{S}$.
      
      Suppose that $H \neq 0$. Then by assumption there exists an internal direct sum decomposition $H=G \oplus J$ where $G \in \mathcal{E}$. Let $\mathcal{B}=\mathcal{A} \cup \{G\}$. Then we have $(\mathcal{B}, J) \in \mathbf{S}$ and $(\mathcal{A}, H) < (\mathcal{B}, J)$, a contradiction to the maximality of $(\mathcal{A}, H)$. 
      
      Therefore $H=0$. Hence $F = \bigoplus_{G \in \mathcal{A}} G$ is a desired decomposition.
   \end{proof}

\section{Zigzag Posets}\label{sec:zigzag-posets}

   In this section, we introduce zigzag persistence modules by first providing a precise definition of a zigzag poset. We then present several observations and definitions that will be essential for working with zigzag persistence modules throughout the article, including the projective colimit conditions (PCC), which underpin some of our main results.

   %%% Notation: Gamma Sets %%%
   \begin{Notation}\label{not:gamma}
      We write $\Gamma$ for one of the following sets: 
      \begin{enumerate}
         \item a finite set $\{1,2, \ldots, n\}$ for $n \geq 2$; 
         \item the set of whole numbers $\mathbb{N}=\{0,1,2,\ldots\}$; or 
         \item the set of integers $\Z$.
      \end{enumerate} 
   \end{Notation}

   The following definition of a zigzag poset is taken from \cite{gan25}*{Definition 7.2}.

   %%% Definition: Zigzag Poset and Extrema %%%
   \begin{Definition} \label{def:zigzag-poset}
      Let $P$ be a poset. Let $\{z_i \mid i\in \Gamma\}$ be a sequence of elements of $P$ indexed by $\Gamma$ (see Notation~\ref{not:gamma}). 
      We call $P$ a \defn{zigzag poset with extrema} $\{z_i \mid i\in \Gamma\}$ if there exists a total order $\sqsubseteq$ on the set $P$ satisfying the following conditions:
      \begin{enumerate}
         \item for all $i,j\in \Gamma$:
         \[ i< j \quad \Longrightarrow \quad z_i\sqsubset z_j; \]
         \item for all $x,y\in P$:
         \begin{align*}
            x\leq y \quad \Longrightarrow \quad   
            & \mbox{ there exists $i\in \Gamma$ such that } \\
            & \begin{cases} 
                  i+1\in \Gamma, \\
                  z_i \sqsubseteq x\sqsubseteq z_{i+1}, \\ 
                  z_i \sqsubseteq y\sqsubseteq z_{i+1};
               \end{cases} 
         \end{align*}
      \item either:
      \begin{align*}
         &\mbox{ for all $i\in \Gamma$ such that $i+1\in \Gamma$, }\\ 
         &\mbox{ for all $x,y\in P$ such that $z_i\sqsubseteq x \sqsubseteq y \sqsubseteq z_{i+1}$: }\\ 
         &\begin{cases} 
            i \mbox{ is odd } \quad \Longrightarrow \quad x\leq y, \\
            i \mbox{ is even } \quad \Longrightarrow \quad  y\leq x,
         \end{cases}  
      \end{align*} 
      or:
      \begin{align*}
         &\mbox{ for all $i\in \Gamma$ such that $i+1\in \Gamma$, }\\
         &\mbox{ for all $x,y\in P$ such that $z_i \sqsubseteq x \sqsubseteq y \sqsubseteq z_{i+1}$: }\\ 
         &\begin{cases}
            i \mbox{ is odd } \quad \Longrightarrow \quad y\leq x , \\
            i \mbox{ is even } \quad \Longrightarrow \quad x\leq y.
         \end{cases}
      \end{align*}   
      \end{enumerate}
      In this case, we call $\sqsubseteq$ the \defn{associated total order}.
   \end{Definition}

   %%% Lemma: Zigzag Shape %%%
   \begin{Lemma} \label{lem:zigzag-shape}
      Let $P$ be a zigzag poset with extrema $\{z_i \mid i\in \Gamma\}$. 
      Let $\sqsubseteq$ be the associated total order. Then we have the following. 
      \begin{enumerate}
         \item For all $x\in P$, there exists $i\in \Gamma$ such that $i+1\in \Gamma$ and $z_i\sqsubseteq x\sqsubseteq z_{i+1}$.
      
         \item For all $x,y\in P$ such that $x\leq y$, for all $i\in \Gamma$ such that  
         \[ x\sqsubseteq z_i \sqsubseteq y \quad \mbox{ or } \quad y\sqsubseteq z_i \sqsubseteq x, \] 
         we have $x=z_i$ or $y=z_i$.

         \item For all $x,y,z\in P$ such that $x\leq y\leq z$, we have 
         \[ x\sqsubseteq y\sqsubseteq z \quad \mbox{ or } \quad z\sqsubseteq y \sqsubseteq x. \]
      \end{enumerate}
   \end{Lemma}

   \begin{proof}
      See \cite{gan25}*{Lemma 7.3}.
   \end{proof}

   %%% Lemma: Interval Shape in Zigzag %%%
   \begin{Lemma} \label{lem:interval-in-zigzag}
      Let $P$ be a zigzag poset with extrema $\{z_i \mid i\in \Gamma\}$. 
      Let $\sqsubseteq$ be the associated total order. Let $I$ be a nonempty subset of $P$. Then the following statements are equivalent:
      \begin{enumerate}
         \item $I$ is an interval in the poset $P$.
         \item For all $x,y,z\in P$ such that $x\sqsubseteq y\sqsubseteq z$, we have $y\in I$ whenever $x, z\in I$.
      \end{enumerate}
   \end{Lemma}

   \begin{proof}
      See \cite{gan25}*{Lemma 7.4}
   \end{proof}

   %%% Remark: Lemmas on Zigzag Behavior %%%
   The above lemmas, in effect, guarantee that our definition of a zigzag poset behaves in the expected way. 

   %%% Notation: Intervals in Zigzag %%%
   \begin{Notation}\label{not:zigzag-interval}
      Let $P$ be a zigzag poset with extrema $\{z_i \mid i\in \Gamma\}$. 
      Let $\sqsubseteq$ be the associated total order. Given $x \sqsubseteq y$ in $P$, we use the notation below to represent the following intervals in $P$ (when such subsets are nonempty):

      \begin{enumerate}
         \item $[x,y] = \{z \in P \mid x \sqsubseteq z \sqsubseteq y\}$;
         \item $[x,y) = \{z \in P \mid x \sqsubseteq z \sqsubset y\}$;
         \item $(x,y] = \{z \in P \mid x \sqsubset z \sqsubseteq y\}$;
         \item $(x,y) = \{z \in P \mid x \sqsubset z \sqsubset y\}$.
      \end{enumerate}

      In particular, when $x=y$ we have $[x,x] = \{x\}$.
   \end{Notation}

   %%% Definition: Zigzag Persistence Module %%%
   \begin{Definition}\label{def:zigzag-pers-mod}
      Let $P$ be a zigzag poset with extrema $\{z_i \mid i\in \Gamma\}$. A functor $F:P \to \RMod$ is a \defn{zigzag persistence module with extrema} $\{z_i \mid i\in \Gamma\}$, or more generally a \defn{zigzag persistence module}.
   \end{Definition}

   %%% Defintion: Projective Colimit Conditions %%% 
   \begin{Definition}\label{def:proj-colim-cond}
      Let $P$ be a zigzag poset with extrema $\{z_i \mid i\in \Gamma\}$. Let $\sqsubseteq$ be the associated total order. We say a zigzag persistence module $F:P \to \RMod$ satisfies the \defn{projective colimit conditions (PCC)} if the following conditions hold for all indices $x \sqsubseteq y$ in $P$: 

      \begin{enumerate}
         \item[\textbf{(C1)}] $\Colim(F|_{[x,y]})$ is projective; 
         \item[\textbf{(C2)}] $\Coker(F_x \to \Colim(F|_{[x,y]}))$ is projective; 
         \item[\textbf{(C3)}] $\Coker(F_y \to \Colim(F|_{[x,y]}))$ is projective;
         \item[\textbf{(C4)}] $\Coker(F_x \oplus F_y \to \Colim(F|_{[x,y]}))$ is projective. 
      \end{enumerate}
   \end{Definition}

   %%% Lemma: Other Facts of PCC %%%
   \begin{Lemma}\label{lem:pcc-more-facts}
      Let $F:P \to \RMod$ be a p.f.g. zigzag persistence module with extrema $\{z_i \mid i\in \Gamma\}$ satisfying the PCC. Let $\leq$ be the partial order and $\sqsubseteq$ be the associated total order. Then:
      \begin{enumerate}
         \item $F_x$ is a finitely generated projective $R$-module for every $x \in P$;
         \item $\Coker(F_{x,y})$ is projective for every $x \leq y$ in $P$; 
         \item Any direct summand of $F$ also satisfies the PCC.
      \end{enumerate}
   \end{Lemma}

   \begin{proof}
      (1) Immediate from \textbf{(C1)} when $x=y$. 
      
      (2) If $x \leq y$ and $x \sqsubseteq y$, then $\Colim(F|_{[x,y]})=F_y$, hence \textbf{(C2)} implies $\Coker(F_{x,y})$ projective. Similarly, if $x \leq y$ and $y \sqsubseteq x$, then $\Colim(F|_{[y,x]})=F_y$, hence \textbf{(C3)} implies $\Coker(F_{x,y})$ is projective.
      
      (3) This follows from the fact that colimits commute with coproducts and cokernels in the category $\RMod$ \cite{par70}*{\S2.7, Corollary 2}. In particular, it is essential that $F$ is pointwise finitely generated so that if $F=\oplus_{G \in \mathcal{A}}G$ is a decomposition of $F$, then only finitely many $G_x$ are nontrivial for any fixed $x \in P$. 
   \end{proof}

   To end this section, we introduce the notion of a peak of a zigzag persistence module, along with some useful observations regarding peaks that will be used later in this article. Our definition is adapted from Ringel's work on the representation theory of Dynkin quivers \cite{rin16}.

   %%% Defintion: Peak %%%
   \begin{Definition}\label{def:zigzag-peak}
      Let $F:P \to \RMod$ be a zigzag persistence module with extrema $\{z_i \mid i\in \Gamma\}$. Let $\sqsubseteq$ be the associated total order. We call an index $x \in P$ a \defn{peak} of $F$ provided that for all morphisms $y \to z$:
      \begin{enumerate}
         \item the map $F_{y,z}$ is injective whenever $y \sqsubseteq z \sqsubseteq x$ or $x \sqsubseteq z \sqsubseteq y$;
         \item the map $F_{y,z}$ is surjective whenever $z \sqsubseteq y \sqsubseteq x$ or $x \sqsubseteq y \sqsubseteq z$.
      \end{enumerate}
   \end{Definition}

   %%%% Lemma: Peak Facts %%%
   \begin{Lemma}\label{lem:peak-facts}
      Let $P$ be a zigzag poset with extrema $\{z_i \mid i\in \Gamma\}$.
      \begin{enumerate}
         \item If $F:P \to \RMod$ is a zigzag persistence module of projective constant type over an interval $I$ in $P$, then $F$ has a peak at every $x \in I$. 
         \item A direct sum of zigzag persistence modules indexed by $P$ with $x$ as a peak also has $x$ as a peak. 
         \item If $F:P \to \RMod$ is a zigzag persistence module with $x$ as a peak, then any direct summand of $F$ also has $x$ as a peak. 
      \end{enumerate}
   \end{Lemma}

   \begin{proof}
      The above statements are easily verified from Definition~\ref{def:zigzag-peak} and the properties of injective and surjective maps with respect to direct sums and direct summands.
   \end{proof}

\section{\texorpdfstring{$A_n$}{TEXT}-Persistence Modules}\label{sec:quiver-zigzag}

   An ${A}_n$-persistence module (when $n \geq 2$) is a special case of a zigzag persistence module with finite extrema (see Definition~\ref{def:zigzag-pers-mod}). The results on projective constant decompositions of ${A}_n$-persistence modules developed in this section play a critical role in establishing subsequent results for totally ordered persistence modules in Section~\ref{sec:total-order} and zigzag persistence modules in Sections~\ref{sec:finite-zigzag} and \ref{sec:infinite-zigzag}. For this reason, we devote an entire section to this special case. Moreover, $A_n$-persistence modules bear a strong resemblance to $A_n$-type quiver representations, albeit with additional structure arising from the persistence setting.

   %%% Definition: A_n Quiver %%%
   \begin{Definition}
      Given a quiver of type $A_n$ with the simply laced Dynkin diagram

      
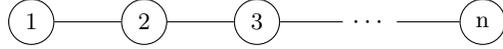
\begin{center}
         \begin{tikzpicture}[roundnode/.style={circle,draw=black, minimum size=6mm}]
            \node[roundnode] at (-3,0) (node1) {\small 1};
            \node[roundnode] at (-1.5,0) (node2) {\small 2};
            \node[roundnode] at (0,0) (node3) {\small 3};
            \node[] at (1.5,0) (node4) {$\cdots$};
            \node[roundnode] at (3,0) (node5) {\small n};
            \draw (node1) -- (node2);
            \draw (node2) -- (node3);
            \draw (node3) -- (node4);
            \draw (node4) -- (node5);
         \end{tikzpicture}
      \end{center}

      we can view $A_n$ as a poset with elements $\{1,2, \ldots, n\}$ and a partial order denoted by $\leq$. A functor $F:A_n \to \RMod$ is an \defn{$A_n$-persistence module}. It is easy to see that an $A_n$-persistence module with $n \geq 2$ is in fact a zigzag persistence module with finite extrema and whose associated total order $\sqsubseteq$ is the natural total ordering of the integer indices (see Definitions~\ref{def:zigzag-poset} and \ref{def:zigzag-pers-mod}).
   \end{Definition}

   %%% Definition: A_n Projective Colimit Conditions %%%
   \begin{Definition}
      Let $F:A_n \to \RMod$ be an $A_n$-persistence module. We say $F$ satisfies the \defn{projective colimit conditions (PCC)} if:
      \begin{enumerate}
         \item When $n=1$: $F_1$ is a projective $R$-module; 
         \item When $n \geq 2$: $F$ satisfies the PCC when viewed as a zigzag persistence module (see Definition~\ref{def:proj-colim-cond}).
      \end{enumerate}
   \end{Definition}

   %%% Remark: PCC Conditions in A_n Case %%%
   Some care is required in the above definition, since an $A_1$-persistence module is not, strictly speaking, a zigzag persistence module. However, the conditions align in the obvious way to those provided in Definition~\ref{def:proj-colim-cond}.

   We now prove the main result of this section, namely Theorem~\ref{thm:main-pc-quiver}. The proof is modeled after Ringel's decomposition of representations of $A_n$-type quivers \cite{rin16}.

   %%% Theorem: A_n Zigzag Persistence Module Admits Projective Constant Decomposition %%%
   \begin{Theorem}\label{thm:quiver-pc-decomp}
      Let $F:A_n \to \RMod$ be a nonzero p.f.g. $A_n$-persistence module satisfying the PCC. Then $F$ admits a projective constant decomposition.
   \end{Theorem}

   \begin{proof}
      %%% Case for N=1 and N=2 %%%
      \begingroup
         We will proceed by induction on $n$. The case for $n=1$ is a single projective $R$-module $F_1$, hence trivially of projective constant type. 
         
         For $n=2$, this is the case with either one map $F_{1,2}:F_1 \to F_2$ or $F_{2,1}:F_2 \to F_1$. Without loss of generality, consider the case $F_{1,2}:F_1 \to F_2$. Since $\Coker[1,2]$ is projective by Lemma~\ref{lem:pcc-more-facts}, $F_2 = \Ima[1,2] \oplus G_2$ for some $R$-submodule $G_2$. Moreover, $F_1/\Ker[1,2] \cong \Ima[1,2]$ is projective since $\Ima[1,2]$ is a direct summand of $F_2$. Hence $F_1 = \Ker[1,2] \oplus G_1$ for some $R$-submodule $G_1 \cong \Ima[1,2]$. We decompose $F$ in the following way.
         \begin{center}
            \begin{tikzcd}[row sep=-5pt, column sep=25pt]
               F_1 & F_2 \\
               \rotatebox{90}{=} & \rotatebox{90}{=} \\
               \Ker[1,2] & \\
               \oplus &  \\
               G_1 \arrow[r, "\cong"] & \Ima[1,2] \\
               & \oplus \\
               & G_2 \\
            \end{tikzcd}
         \end{center}
         Each row in the diagram represents a submodule of $F$ of projective constant type. Hence $F$ admits a projective constant decomposition.
      \endgroup

      Now let $n \geq 3$ and assume the inductive hypothesis holds up to $n-1$. Let $\sqsubseteq$ be the associated total order when viewing $F$ as a zigzag persistence module. 
         
      %%% Claim 7: Decomposition is Compatible with Chain of Images %%%
      \textbf{Claim~\ref*{thm:quiver-pc-decomp}.1.} For any fixed $x$ such that $1 \sqsubset x \sqsubset n$, the persistence module $F$ can be written as $F=G \oplus H \oplus J$, where $H$ has $x$ as a peak (see Definition~\ref{def:zigzag-peak}), the support of $G$ is in $\{1, \ldots, x-1\}$, and the support of $J$ is in $\{x+1, \ldots, n\}$.

      \begin{proof}[Proof of Claim~\ref*{thm:quiver-pc-decomp}.1.] 	
         Fix $1 \sqsubset x \sqsubset n$. By the inductive hypothesis, $F|_{[1,x]} = G' \oplus H'$, where $G'$ and $H'$ decompose as a direct sum of persistence modules of projective constant type, $G'_x=0$, and $H'$ has $x$ as a peak (see Lemma~\ref{lem:peak-facts}). 
         
         Likewise, by the inductive hypothesis, $F|_{[x,n]} = H'' \oplus J''$ where $H''$ and $J''$ decompose as a direct sum of persistence modules of projective constant type, $J''_x=0$, and $H''$ has $x$ as a peak. Since $G'_x=0=J''_x$, we must have $F_x=H'_x=H''_x$. Define the persistence submodules $G,H,$ and $J$ as follows:
         \begin{align*}
            G_y = 
            \begin{cases} 
               G'_y & \text{if } y \sqsubset x, \\
               0 & \text{if } y \sqsupseteq x; \\
            \end{cases} 
            \hspace{15pt} 
            & H_y = 
            \begin{cases} 
               H'_y & \text{if } y \sqsubset x, \\
               H''_y &  \text{if } y \sqsupseteq x; \\
            \end{cases}
            & J_y =
            \begin{cases} 
               0 & \text{if } y \sqsubset x, \\
               J''_y &  \text{if } y \sqsupseteq x. \\
            \end{cases}
         \end{align*}

         Then $F=G \oplus H \oplus J$ is the desired decomposition.
      \end{proof}

      %%% Reducing to the case when N=3 %%%
      \begingroup
         Applying Claim~\ref*{thm:quiver-pc-decomp}.1 to $F$ for $x=2$, we can write $F=G \oplus H \oplus J$, where $H$ has a peak at $2$, the support of $G$ is in $\{1\}$, and the support of $J$ is in $\{3, \ldots, n\}$. 
         
         Applying Claim~\ref*{thm:quiver-pc-decomp}.1 again to $H$ for $x=n-1$, we can write $H=\tilde{G} \oplus \tilde{H} \oplus \tilde{J}$ where $\tilde{H}$ has a peak at $n-1$, the support of $\tilde{G}$ is in $\{1, \ldots, n-2\}$, and the support of $\tilde{J}$ is in $\{n\}$. In particular, note that $\tilde{H}$ also has a peak at $2$ as it is a direct summand of $H$ (see Lemma~\ref{lem:peak-facts}). 
         
         This yields a decomposition $F=G \oplus \tilde{G} \oplus \tilde{H} \oplus \tilde{J} \oplus J$. Now $G, \tilde{G}, J$ and $\tilde{J}$ satisfy the PCC by Lemma~\ref{lem:pcc-more-facts}, and hence all admit projective constant decompositions by the inductive hypothesis. It remains to show that $\tilde{H}$ admits a projective constant decomposition, where $\tilde{H}$ has peaks at $2$ and $n-1$. 
         
         When $n \geq 4$, all the internal morphisms of $\tilde{H}$ between $2$ and $n-1$ are, in fact, isomorphisms. Therefore we reduce to the case of showing a projective constant decomposition exists for any $A_3$-persistence module $F$ with a peak at $2$. To do this, we consider three separate cases based on the possible configurations of morphisms in $F$.
      \endgroup
      
      %%% N=3 Case #1 %%%
      \begingroup
         \textbf{Case 1.} Consider the totally ordered $A_3$-persistence modules with the following diagrams:
         \begin{center}
            \begin{tikzcd}
               F_1 \arrow[r, "F_{1,2}", hook] \arrow[rr, "F_{1,3}", bend left=45]& F_2 \arrow[r, "F_{2,3}", two heads]& F_3
            \end{tikzcd} or
            \begin{tikzcd}
               F_1 & F_2 \arrow[l, "F_{2,1}", two heads, swap]& F_3 \arrow[ll, "F_{3,1}", bend right=45, swap] \arrow[l, "F_{3,2}", swap, hook].
            \end{tikzcd}
         \end{center}
         Without loss of generality, we will prove the results for the left diagram. Note that $F_{1,2}$ is injective and $F_{2,3}$ is surjective since $F$ has a peak at $2$. 
         
         Observe from Lemma~\ref{lem:pcc-more-facts} that $\Coker[1,3] = F_3/\Ima[1,3] \cong F_2 / (\Ima[1,2] + \Ker[2,3])$ is projective, $F_2 / \Ker[2,3] \cong F_3$ is projective, and $\Coker[1,2] = F_2 / \Ima[1,2]$ is projective. We have the following short exact sequences:
         \begin{center}
            \begin{tikzcd}[row sep = 2pt, column sep = 20pt]
               0 \arrow [r] & \Frac{\Ima[1,2]}{\Ima[1,2] \cap \Ker[2,3]} \arrow[r, hook] & \Frac{F_2}{\Ker[2,3]} \arrow[r, two heads]& \Frac{F_2}{\Ima[1,2]+\Ker[2,3]} \arrow[r] & 0; \\
               0 \arrow [r] & \Frac{\Ker[2,3]}{\Ima[1,2] \cap \Ker[2,3]} \arrow[r, hook] & \Frac{F_2}{\Ima[1,2]} \arrow[r, two heads]& \Frac{F_2}{\Ima[1,2]+\Ker[2,3]} \arrow[r] & 0.
            \end{tikzcd}
         \end{center}
         Both sequences split, implying $\Ima[1,2]/(\Ima[1,2] \cap \Ker[2,3])$ and $\Ker[2,3]/(\Ima[1,2] \cap \Ker[2,3])$ are projective. Furthermore, we  have the following short exact sequences:
         \begin{center}
            \begin{tikzcd}[row sep = 2pt, column sep = 20pt]
               0 \arrow [r] & \Ima[1,2] \cap \Ker[2,3] \arrow[r, hook] & \Ima[1,2] \arrow[r, two heads]& \Frac{\Ima[1,2]}{\Ima[1,2] \cap \Ker[2,3]} \arrow[r] & 0; \\
               0 \arrow [r] & \Ima[1,2] \cap \Ker[2,3] \arrow[r, hook] & \Ker[2,3] \arrow[r, two heads]& \Frac{\Ker[2,3]}{\Ima[1,2] \cap \Ker[2,3]} \arrow[r] & 0. \\
            \end{tikzcd}
         \end{center}
         Both sequences split, yielding the following decompositions and isomorphisms for some $R$-submodules $G,K$ and $H$:
         \begin{flalign*}
            & \Ima[1,2]=(\Ima[1,2] \cap \Ker[2,3]) \oplus G; \\
            & \Ker[2,3] = (\Ima[1,2] \cap \Ker[2,3]) \oplus K; \\
            & \Ima[1,2]+\Ker[2,3]=(\Ima[1,2] \cap \Ker[2,3]) \oplus G \oplus K; \\
            & F_2= (\Ima[1,2] \cap \Ker[2,3]) \oplus G \oplus K \oplus H; \\
            & F_3 \cong H \oplus G; \\
            & F_1 \cong \Ima[1,2]. 
         \end{flalign*}

         Hence, we can decompose $F$ in the following way.
         \begin{center}
            \begin{tikzcd}[row sep=-5pt, column sep=25pt]
               F_1 \arrow[r, "F_{1,2}", hook]& F_2 \arrow[r, "F_{2,3}", two heads] & F_3 \\
               \rotatebox{90}{=} & \rotatebox{90}{=} & \rotatebox{90}{=} \\
               &&\vspace{2in}   \\ 
               &H \arrow[r, "\cong"]& F_{2,3}(H) \\
               & \oplus & \oplus \\
               \Inv{F_{1,2}}(G) \arrow[r, "\cong"] & G \arrow[r, "\cong"] & F_{2,3}(G) \\
               \oplus & \oplus & \\
                  \Inv{F_{1,2}}(\Ima[1,2] \cap \Ker[2,3]) \arrow[r, "\cong"] & \Ima[1,2] \cap \Ker[2,3] & \\
               & \oplus & \\
               &K& \\
            \end{tikzcd}
         \end{center}

         Each row in the diagram represents a submodule of $F$ of projective constant type. Hence $F$ admits a projective constant decomposition.
      \endgroup
      
      %%% N=3 Case #2 %%%
      \begingroup
         \textbf{Case 2.} Consider the $A_3$-persistence module with the following diagram.

         \begin{center}
            \begin{tikzcd}
               F_1 \arrow[r, "F_{1,2}", hook] & F_2 & F_3  \arrow[l, "F_{3,2}", swap, hook']
            \end{tikzcd}
         \end{center}

         Note that $F_{1,2}$ and $F_{3,2}$ are both injective since $F$ has a peak at $2$. We have $\Colim(F|_{[1,3]}) = F_2$, and also $\Coker[1,2]$ and $\Coker[3,2]$ are projective by Lemma~\ref{lem:pcc-more-facts}. Moreover,
         \begin{center}
            $\Coker(F_1 \oplus F_3 \to F_2) = \Frac{F_2}{\Ima[1,2]+\Ima[3,2]}$ 
         \end{center}
         and is projective by \textbf{(C4)}. We have the following short exact sequences:

         \begin{center}
            \begin{tikzcd}[row sep=2pt, column sep=20pt]
               0 \arrow [r] & \Frac{\Ima[1,2]}{\Ima[1,2] \cap \Ima[3,2]} \arrow[r, hook] & \Frac{F_2}{\Ima[3,2]} \arrow[r, two heads]& \Frac{F_2}{\Ima[1,2]+\Ima[3,2]} \arrow[r] & 0; \\
               0 \arrow [r] & \Frac{\Ima[3,2]}{\Ima[1,2] \cap \Ima[3,2]} \arrow[r, hook] & \Frac{F_2}{\Ima[1,2]} \arrow[r, two heads]& \Frac{F_2}{\Ima[1,2]+\Ima[3,2]} \arrow[r] & 0. \\
            \end{tikzcd}
         \end{center}

         Both sequences split, implying $\Ima[1,2]/(\Ima[1,2] \cap \Ima[3,2])$ and $\Ima[3,2]/(\Ima[1,2] \cap \Ima[3,2])$ are projective. Furthermore, we have the following short exact sequences:

         \begin{center}
            \begin{tikzcd}[row sep = 2pt, column sep=20pt]
               0 \arrow [r] & \Ima[1,2] \cap \Ima[3,2] \arrow[r, hook] & \Ima[1,2] \arrow[r, two heads]& \Frac{\Ima[1,2]}{\Ima[1,2] \cap \Ima[3,2]} \arrow[r] & 0; \\
               0 \arrow [r] & \Ima[1,2] \cap \Ima[3,2] \arrow[r, hook] & \Ima[3,2] \arrow[r, two heads]& \Frac{\Ima[3,2]}{\Ima[1,2] \cap \Ima[3,2]} \arrow[r] & 0. \\
            \end{tikzcd}
         \end{center}

         Both sequences split, yielding the following decompositions and isomorphisms for some $R$-submodules $G,K$ and $H$:
         \begin{flalign*}
            & \Ima[1,2]=(\Ima[1,2] \cap \Ima[3,2]) \oplus G; \\
            & \Ima[3,2] = (\Ima[1,2] \cap \Ima[3,2]) \oplus K; \\
            & \Ima[1,2]+\Ima[3,2]=(\Ima[1,2] \cap \Ima[3,2]) \oplus G \oplus K; \\
            & F_2=(\Ima[1,2]+\Ima[3,2]) \oplus H = (\Ima[1,2] \cap \Ima[3,2]) \oplus G \oplus K \oplus H; \\
            & F_1 \cong \Ima[1,2];\\
            & F_3 \cong \Ima[3,2].
         \end{flalign*}

         Hence, we can decompose $F$ in the following way.

         %%% Final Decomposition %%%
         \begin{center}
            \begin{tikzcd}[row sep=-5pt, column sep=25pt]
               F_1 \arrow[r, "F_{1,2}", hook]& F_2 & F_3 \arrow[l, "F_{3,2}", swap, hook'] \\
               \rotatebox{90}{=} & \rotatebox{90}{=} & \rotatebox{90}{=} \\
               &&\vspace{2in}   \\ 
               &H& \\
               & \oplus & \\
               \Inv{F_{1,2}}(G) \arrow[r, "\cong"] & G &  \\
               \oplus & \oplus & \\
               \Inv{F_{1,2}}(\Ima[1,2] \cap \Ima[3,2]) \arrow[r, "\cong"] & \Ima[1,2] \cap \Ima[3,2] &  F_{3,2}^{-1}(\Ima[1,2] \cap \Ima[3,2])  \arrow[l, "\cong", swap] \\
               & \oplus & \oplus \\
               &K& F_{3,2}^{-1}(K) \arrow[l, "\cong", swap]\\
            \end{tikzcd}
         \end{center}
         
         Each row in the diagram represents a submodule of $F$ of projective constant type. Hence $F$ admits a projective constant decomposition.
      \endgroup

      %%% N=3 Case #3 %%%
      \begingroup
         \textbf{Case 3.} Consider the $A_3$-persistence module with the following diagram.

         \begin{center}
            \begin{tikzcd}
               F_1 & \arrow[l, "F_{2,1}", swap, two heads] F_2 \arrow[r, "F_{2,3}", two heads] & F_3
            \end{tikzcd}
         \end{center}

         Note that $F_{2,1}$ and $F_{2,3}$ are both surjective since $F$ has a peak at $2$. We have $\Colim(F|_{[1,3]}) = (F_1 \oplus F_3)/E$ where $E=\Set{(F_{2,1}(a), -F_{2,3}(a))}{a \in F_2}$. By Lemma~\ref{lem:pcc-more-facts}, $F_2/\Ker[2,1] \cong \Ima[2,1]=F_1$ is projective, and similarly $F_2 / \Ker[2,3] \cong \Ima[2,3]=F_3$ is projective. Moreover, 
         \begin{center}
            $\Frac{F_1 \oplus F_3}{E} \cong \Frac{F_2} {\Ker[2,1]+\Ker[2,3]}$ 
         \end{center}
         and is projective by \textbf{(C1)}. We have the following short exact sequences:
         \begin{center}
            \begin{tikzcd}[row sep=2pt, column sep=20pt]
               0 \arrow [r] & \Frac{\Ker[2,1]}{\Ker[2,1] \cap \Ker[2,3]} \arrow[r, hook] & \Frac{F_2}{\Ker[2,3]} \arrow[r, two heads]& \Frac{F_2}{\Ker[2,1]+\Ker[2,3]} \arrow[r] & 0; \\
               0 \arrow [r] & \Frac{\Ker[2,3]}{\Ker[2,1] \cap \Ker[2,3]} \arrow[r, hook] & \Frac{F_2}{\Ker[2,1]} \arrow[r, two heads]& \Frac{F_2}{\Ker[2,1]+\Ker[2,3]} \arrow[r] & 0. \\
            \end{tikzcd}
         \end{center}
         Both sequences split, implying $\Ker[2,1]/(\Ker[2,1] \cap \Ker[2,3])$ and $\Ker[2,3]/(\Ker[2,1] \cap \Ker[2,3])$ are projective. Furthermore, we have the following short exact sequences:
         \begin{center}
            \begin{tikzcd}[row sep=2pt, column sep=20pt]
               0 \arrow [r] & \Ker[2,1] \cap \Ker[2,3] \arrow[r, hook] & \Ker[2,1] \arrow[r, two heads]& \Frac{\Ker[2,1]}{\Ker[2,1] \cap \Ker[2,3]} \arrow[r] & 0; \\
               0 \arrow [r] & \Ker[2,1] \cap \Ker[2,3] \arrow[r, hook] & \Ker[2,3] \arrow[r, two heads]& \Frac{\Ker[2,3]}{\Ker[2,1] \cap \Ker[2,3]} \arrow[r] & 0. \\
            \end{tikzcd}
         \end{center}

         Both sequences split, yielding the following decompositions and isomorphisms for some $R$-submodules $G, K$ and $H$: 
         \begin{flalign*}
            &\Ker[2,1]=(\Ker[2,1] \cap \Ker[2,3]) \oplus G; \\
            &\Ker[2,3]=(\Ker[2,1] \cap \Ker[2,3]) \oplus K; \\
            &\Ker[2,1]+\Ker[2,3]=(\Ker[2,1] \cap \Ker[2,3]) \oplus G \oplus K; \\
            &F_2=(\Ker[2,1]+\Ker[2,3]) \oplus H = (\Ker[2,1] \cap \Ker[2,3]) \oplus G \oplus K \oplus H; \\
            &F_1 \cong F_2/ \Ker[2,1] \cong K \oplus H; \\
            &F_3 \cong F_2 / \Ker[2,3] \cong G \oplus H.
         \end{flalign*}
         
         Hence, we can decompose $F$ in the following way.

         %%% Final Decomposition %%%
         \begin{center}
            \begin{tikzcd}[row sep=-5pt, column sep=25pt]
               F_1 & \arrow[l, "F_{2,1}", two heads, swap] F_2  \arrow[r, "F_{2,3}", swap, two heads, swap] & F_3 \\
               \rotatebox{90}{=} & \rotatebox{90}{=} & \rotatebox{90}{=} \\
               &&\vspace{2in}   \\ 
               &G \arrow[r, "\cong"]& F_{2,3}(G) \\
               & \oplus & \oplus \\
               F_{2,1}(H) & H \arrow[l, "\cong", swap] \arrow [r, "\cong"]& F_{2,3}(H) \\
               \oplus &  \oplus & \\
               F_{2,1}(K) & K \arrow[l, "\cong"]& \\
               & \oplus & \\
               & \Ker[2,1] \cap \Ker[2,3] & \\
            \end{tikzcd}
         \end{center}
         
         Each row in the diagram represents a submodule of $F$ of projective constant type. Hence $F$ admits a projective constant decomposition.
      \endgroup
   \end{proof}

   We deduce Corollary~\ref{cor:main-pid-quiver} in the following result.

   %%% Corollary: PID Zigzag Case %%%
   \begin{Corollary}\label{cor:quiver-pid-interval-decomp}
      Let $R$ be a Noetherian ring with identity such that every finitely generated projective $R$-module is free. Let $F:A_n \to \RMod$ be a nonzero p.f.g. $A_n$-persistence module satisfying the PCC. Then $F$ is interval decomposable.
   \end{Corollary}

   \begin{proof}
      Immediate from Theorem~\ref{thm:quiver-pc-decomp} and Lemma~\ref{lem:free-interval-decomp}.
   \end{proof}

\section{Totally Ordered Persistence Modules}\label{sec:total-order}

   Throughout this section, we now assume that $R$ is a commutative Noetherian ring with identity, unless otherwise specified. Our goal is to develop decomposition results for totally ordered persistence modules, which will later be necessary in the decomposition of zigzag persistence modules in Sections~\ref{sec:finite-zigzag} and~\ref{sec:infinite-zigzag}. In particular, Lemma~\ref{lem:pcc-more-facts}(2) showed that, under the assumption of the projective colimit conditions, the cokernel of every internal morphism of a zigzag persistence module is projective. Motivated by this result, we identify the projectivity of the cokernels of internal morphisms as the key structural condition ensuring the existence of a projective constant decomposition.

   We begin by establishing several technical lemmas concerning the existence of compatible chains of images and kernels, along with results on dual spaces and annihilators of projective $R$-modules. These tools are then used to prove Proposition~\ref{prp:total-one-pc-decomp}, which shows that any totally ordered persistence module admits a decomposition containing a single direct summand of projective constant type. Finally, we prove Theorem~\ref{thm:total-complete-pc-decomp}, which guarantees a complete projective constant decomposition of a totally ordered persistence module.

   %%% Definition: Totally Ordered Persistence Module %%%
   \begin{Definition}
      Let $P$ be a totally ordered poset. Let $\leq$ denote the total ordering of $P$. A functor $F:P \to \RMod$ is a \defn{totally ordered persistence module}. 
   \end{Definition}

   %%% Definition: Compatible Chain %%%
   \begin{Definition}\label{def:compatible-chain}
      Let $M$ be an $R$-module. We say an internal direct sum decomposition $M = \bigoplus_{B \in \mathcal{B}} B$ is \defn{compatible} with a chain $S$ of $R$-submodules of $M$ if for every $A \in S$ we have $A=\bigoplus_{B \in \mathcal{A}} B$ for some subset $\mathcal{A} \subseteq \mathcal{B}$. 
   \end{Definition}

   %%% Definition-Lemma: Chain of Images and Kernels %%%
   \begin{Definition-Lemma}\label{deflem:chain-ima-ker}
      Let $F:P \to \RMod$ be a totally ordered persistence module. For any fixed $z \in P$, the set of images $\{\Ima[x,z]:x \leq z\}$ forms a chain of $R$-submodules of $F_z$, called the \defn{chain of images} of $F_z$. Dually, for any fixed $z \in P$, the set of kernels $\{\Ker[z,y]:y \geq z\}$ forms a chain of $R$-submodules of $F_z$, called the \defn{chain of kernels} of $F_z$. 
   \end{Definition-Lemma}

   \begin{proof}
      For a fixed $z \in P$, observe that either $\Ima[x,z] \subseteq \Ima[y,z]$ when $x \leq y$, or $\Ima[y,z] \subseteq \Ima[x,z]$ when $y \leq x$. Hence $\{\Ima[x,z]:x \leq z\}$ forms a chain of $R$-submodules of $F_z$. The setting for the set of kernels is analogous.
   \end{proof}

   %%% Lemma: Finite Chain of Images and Kernels %%%
   \begin{Lemma}\label{lem:fin-chain-ima-ker}
      Let $F:P \to \RMod$ be a p.f.g. totally ordered persistence module. Suppose that the cokernel of every internal morphism $F_{x,y}$ is projective. Then for any fixed $z \in P$, the chain of images of $F_z$ and the chain of kernels of $F_z$ are both finite sets.
   \end{Lemma}

   \begin{proof}
      Since $F_z/\Ima[x,z]=\Coker[x,z]$ is projective for every $x \leq z$, the chain of images of $F_z$ is  finite by Lemma~\ref{lem:fin-chain}. Likewise, since $F_z/\Ker[z,y] \cong \Ima[z,y]$ is projective for every $y \geq z$, the chain of kernels of $F_z$ is finite by Lemma~\ref{lem:fin-chain}.
   \end{proof}

   %%% Lemma: Decomposition Compatible with Chain of Images and Kernels %%%
   \begin{Lemma}\label{lem:comp-decom-ima-ker}
      Let $F:P \to \RMod$ be a nonzero p.f.g. totally ordered persistence module. Suppose that the cokernel of every internal morphism $F_{x,y}$ is projective. Then for any $F_z \neq 0$, there exists an internal direct sum decomposition $F_z = \bigoplus_{G \in \mathcal{A}} G$ of projective $R$-submodules $G$ that is compatible with the chain of images and the chain of kernels of $F_z$. 
   \end{Lemma}

   \begin{proof}
      Fix $z \in P$ such that $F_z \neq 0$. There is a finite chain of images of $F_z$ and a finite chain of kernels of $F_z$ by Lemma~\ref{lem:fin-chain-ima-ker}. Let $l$ and $m$ be the number of unique images and kernels in these chains, respectively. Choose indices $x_i \leq z$ for $1 \leq i \leq l$ so that each $\Ima[x_i,z]$ is one of the unique images in the chain of images of $F_z$. Likewise, choose indices $y_j \geq z$ for indices $1 \leq j \leq m$ so that each $\Ker[z,y_j]$ is one of the unique kernels in the chain of kernels of $F_z$. Let $S = \{x_i\}_{i=1}^l \cup \{y_j\}_{j=1}^m \cup \{z\}$. Then $F|_S$ is a nonzero $A_n$-persistence module satisfying the PCC, and therefore admits a projective constant decomposition by Theorem~\ref{thm:quiver-pc-decomp}. This decomposition at the index $z$ yields the desired result.
   \end{proof}

   %%% Definition-Lemma: Natural Double Dual Isomorphism %%%
   \begin{Definition-Lemma}\label{deflem:nat-double-dual-isom}
      Let $M$ be a finitely generated projective $R$-module. Then there is a natural isomorphism $\varepsilon:M \to M^{**}$ given by $a \mapsto \hat{a}$, where $\hat{a}:M^* \to R$ is the \defn{evaluation map} such that $\hat{a}(f)=f(a)$ for any $f \in M^*$.
   \end{Definition-Lemma}

   \begin{proof}
      See \cite{lam12}*{Corollary 2.10, Remark 2.11}.
   \end{proof}

   %%% Definition-Lemma: Dual Basis of a Projective Module %%%
   \begin{Definition-Lemma}\label{deflem:dual-basis-proj}
      A finitely generated $R$-module $M$ is projective if and only if there exists a family of elements $\{a_i : 1 \leq i \leq n\} \subseteq M$ and a family of linear functionals $\{f_i: 1 \leq i \leq n\} \subseteq M^*$ such that, for any $a \in M$, 
      \[a = \sum_{i=1}^nf_i(a)a_i.\] 
      We call these two families together a \defn{pair of dual bases} of $M$ and denote it by the set of pairs $\{(a_i, f_i) : 1 \leq i \leq n\}$. Moreover, the set $\{a_i : 1 \leq i \leq n\}$ generates $M$ and the set $\{f_i : 1 \leq i \leq n\}$ generates $M^*$.
   \end{Definition-Lemma}

   \begin{proof}
      See \cite{lam12}*{Lemma 2.9, Remark 2.11}.
   \end{proof}

   We caution the reader that the term \textit{pair of dual bases} does not refer to bases of free $R$-modules in the traditional sense. Rather, it refers to a pair of generating sets of a projective $R$-module and its dual satisfying the conditions of Definition-Lemma~\ref{deflem:dual-basis-proj}. Although this terminology is nonstandard, it is convenient for our purposes and is consistent with the usage in \cite{lam12}. Additionally, while the dual of a left $R$-module is naturally a right $R$-module, the commutativity of $R$ allows us to treat the dual as a left $R$-module. Indeed, this identification is the primary motivation for requiring $R$ to be commutative throughout this section.

   %%% Lemma: Dual Basis of M^* %%%
   \begin{Lemma}\label{lem:dual-basis-of-dual}	
      Let $\{(a_i, f_i):1 \leq i \leq n\}$ be a pair of dual bases of a finitely generated projective $R$-module $M$. Then $\{(f_i, \hat{a_i}):1 \leq i \leq n\}$ is a pair of dual bases of $M^*$. 
   \end{Lemma}

   \begin{proof}
      See \cite{lam12}*{Exercise 2.7}.
   \end{proof}

   %%% Lemma: F^*=G + H' %%%
   \begin{Lemma}\label{lem:dual-dir-sum}
      Let $M=A \oplus B$ be a finitely generated projective $R$-module, and let $\{(a_i,f_i) : 1 \leq i \leq m\}$ and $\{ (b_j,g_j) : 1 \leq j \leq n\}$ be pairs of dual bases of $A$ and $B$, respectively. Define $\bar{f}_i:M \to R$ to be the extension of $f_i:A \to R$ with $\bar{f}_i|_A=f_i$ and $\bar{f}_i|_B=0$. Likewise, define $\bar{g}_j:M \to R$ to be the extension of $g_j:B \to R$ with $\bar{g}_j|_B=g_j$ and $\bar{g}_j|_A=0$. Then:
      \begin{enumerate}
         \item The set $\{(a_i,\bar{f}_i) : 1 \leq i \leq m\} \cup \{(b_j,\bar{g}_j) : 1 \leq j \leq n\}$ is a pair of dual bases of $M$;
         \item We can write $M^*=A' \oplus B'$, where $A'$ is generated by $\{\bar{f}_i: 1 \leq i \leq m\}$ and $B'$ is generated by $\{\bar{g}_j : 1 \leq j \leq n\}$. In particular, $A' \cong A^*$ and $B' \cong B^*$. 
      \end{enumerate}
   \end{Lemma}

   \begin{proof}
      Both claims are easily verified using Definition-Lemma~\ref{deflem:dual-basis-proj} and Lemma~\ref{lem:dual-basis-of-dual}.
   \end{proof}

   %%% Lemma: Compatible Maps with Decomposition %%%
   \begin{Lemma}\label{lem:comp-maps-dual}
      Let $\phi:M_1 \to M_2$ be a homomorphism of finitely generated projective $R$-modules where $M_1=A_1 \oplus B_1$ and $M_2=A_2 \oplus B_2$. Write $M_1^* = A'_1 \oplus B'_1$ and $M_2^* = A'_2 \oplus B'_2$ as in Lemma~\ref{lem:dual-dir-sum}. If $\phi(A_1) \subseteq A_2$ and $\phi(B_1) \subseteq B_2$, then $\phi^*(A'_2) \subseteq A'_1$ and $\phi^*(B'_2) \subseteq B'_1$. In particular, if $\phi$ maps $A_1$ to $A_2$ bijectively, then $\phi^*$ maps $A'_2$ to $A'_1$ bijectively.
   \end{Lemma}

   \begin{proof}
      Let $f_2 \in A'_2$. Then $\phi^*(f_2) = f_2 \circ \phi$ is a linear functional in $M_1^*$. There exist $f_1 \in A'_1, g_1 \in B'_1$ such that $f_2 \circ \phi = f_1 + g_1$. Let $b \in B_1$. Then $f_1(b)=0$. It follows that 

      \begin{center}
         $g_1(b) = f_1(b) + g_1(b) = f_2(\phi(b)) = 0$
      \end{center}

      since $\phi(b) \in B_2$ and $f_2|_{B_2}=0$. Moreover, $g_1(a)=0$ for any $a \in A_1$. Hence $g_1=0$ and $\phi^*(f_2)=f_1$, which shows that $\phi^*(A'_2) \subseteq A'_1$. An analogous argument shows $\phi^*(B'_2) \subseteq B'_1$. 

      Now suppose $\phi$ maps $A_1$ to $A_2$ bijectively. We will show that $\phi^*$ maps $A'_2$ to $A'_1$ bijectively.

      Let $f_2 \in A'_2$ and suppose $\phi^*(f_2)=0$. Then $f_2 \circ \phi = 0$. Let $a_2 \in A_2$. Since $\phi$ maps $A_1$ to $A_2$ surjectively, there exists $a_1 \in A_1$ such that $\phi(a_1)=a_2$. This implies $f_2(a_2)=f_2(\phi(a_1))=0$. Moreover, $f_2(b)=0$ for any $b \in B_2$. Hence $f_2=0$, which shows $\phi^*$ maps $A'_2$ to $A'_1$ injectively. 

      Now let $f_1 \in A'_1$. Define a map $f_2 \in A'_2$ so that $f_2|_{A_2} = f_1 \circ (\phi|_{A_1})^{-1}$ and $f_2|_{B_2}=0$. For any $a \in A_1$, 
      
      \begin{center}
      $\phi^*(f_2)(a)=f_2(\phi(a))=f_1 \circ (\phi|_{A_1})^{-1}(\phi(a))=f_1(a)$. 
      \end{center}

      Moreover, for any $b \in B_1$ we have $\phi^*(f_2)(b) = f_2(\phi(b))=0=f_1(b)$. This shows that $\phi^*(f_2)=f_1$, and hence $\phi^*$ maps $A'_2$ to $A'_1$ surjectively.
   \end{proof}

   %%% Defintion: Annihilator %%%
   \begin{Definition}\label{def:annihilator}
      Given a submodule $A$ of an $R$-module $M$, the \defn{annihilator} of $A$ is the submodule $A^\perp$ of $M^*$ where $A^\perp=\{f \in M^* : f(a)=0 \hspace{5pt} \forall a \in A\}$. 
   \end{Definition}

   %%% Lemma: Annihlators Relationships with Images and Kernels %%%
   \begin{Lemma}\label{lem:dual-ann-ker-ima}
      Let $\phi:M \to N$ be a homomorphism of $R$-modules such that $\Coker(\phi)$ is projective. Then $\Ima(\phi^*)=\Ker(\phi)^\perp$.
   \end{Lemma}

   \begin{proof}
      First, suppose $f \in \Ima(\phi^*)$. Then 
      \begin{flalign*}
         f \in \Ima(\phi^*) 
            & \implies \phi^*(g) = g \circ \phi = f \text{ for some } g \in N^* \\
            & \implies (g \circ \phi)(a) = f(a) = 0 \hspace{5pt} \forall a \in \Ker(\phi) \\
            & \implies f \in \Ker(\phi)^\perp. 
      \end{flalign*}

      Conversely, suppose $f \in \Ker(\phi)^\perp$. Since $\Coker(\phi)$ is projective, there exists a submodule $B \subseteq N$ such that $N=\Ima(\phi) \oplus B$. Likewise, since $M/\Ker(\phi)$ is projective, there exists a decomposition $M=\Ker(\phi) \oplus A$ where $\phi|_A:A \to \Ima(\phi)$ is a bijection. Let $\Inv{(\phi|_A)}:\Ima(\phi) \to A$ be the inverse of this bijection and $\pi:N \to \Ima(\phi)$ be the natural projection map.

      Define $g:N \to R$ to be the map with $g=f \circ \Inv{(\phi|_A)} \circ \pi$. For any $c \in \Ker(\phi)$,
      \begin{center}
         $\phi^*(g)(c) = g(\phi(c)) = 0 = f(c)$
      \end{center}
      since $f \in \Ker(\phi)^\perp$. Moreover, for any $a \in A$ we have
      \begin{center}
         $\phi^*(g)(a) = (f \circ \Inv{(\phi|_A)} \circ \pi) (\phi(a)) = (f \circ \Inv{(\phi|_A)}) (\phi(a)) = f(a)$.
      \end{center}
      Hence $\phi^*(g) = f$, which shows $f \in \Ima(\phi^*)$.
   \end{proof}

   %%% Lemma: Dual of a Direct Sum and Annihilators %%%
   \begin{Lemma}\label{lem:dual-direct-sum-ann}
      Let $M = A \oplus B$ be an $R$-module. Then $M^* = A^\perp \oplus B^\perp$. 
   \end{Lemma}

   \begin{proof}
      Straightforward to verify using Definition~\ref{def:annihilator}.
      \begin{comment}
      Let $f \in M^*$. Define $f_A=f \circ \pi_A$ and $f_B= f \circ \pi_B$, where $\pi_A:M \to A$ and $\pi_B:M \to B$ are the natural projection maps. Then for any $a \in M = A \oplus B$, we have 
      \begin{center}
         $(f_A+f_B)(m)= (f_A+f_B)(a+b) = f(a)+f(b) = f(a+b)=f(m)$. 
      \end{center}
      Hence $A^\perp + B^\perp$ generates $M^*$.

      Now suppose $f \in A^\perp \cap B^\perp$. Then for every $m \in M$, we have $f(m)=f(a+b)=f(a)+f(b)=0$. Hence $f=0$. 
      \end{comment}
   \end{proof}

   %%% Lemma: Cokernel of Dual Map is Projective %%%
   \begin{Lemma}\label{lem:coker-dual-map-proj}
      Let $\phi:M \to N$ be a homomorphism of finitely generated projective $R$-modules such that $\Coker(\phi)$ is projective. Then $\Coker(\phi^*)$ is also projective.
   \end{Lemma}

   \begin{proof}
      Clearly $\Coker(\phi)$ projective implies $\Ima(\phi) \cong M/\Ker(\phi)$ is projective. Hence there exists a submodule $A \subseteq M$ such that $M = \Ker(\phi) \oplus A$. Dualizing yields 

      \begin{center}
         $M^* = \Ker(\phi)^\perp \oplus A^\perp = \Ima(\phi^*) \oplus A^\perp$ 
      \end{center}
      
      by Lemmas~\ref{lem:dual-ann-ker-ima} and \ref{lem:dual-direct-sum-ann}. Observe $A^\perp$ is projective as a direct summand of $M^*$ which is projective. Therefore $\Coker(\phi^*) = M^* / \Ima(\phi^*) \cong A^\perp$ is projective.
   \end{proof}

   %%% Lemma: Decomposition Surjective Map %%%
   \begin{Lemma}\label{lem:decomp-surj-map}
      Let $\phi:M \to N$ be a surjection of $R$-modules such that $N$ is projective. If $N=N_1 \oplus N_2$, then we can choose an internal direct sum decomposition $M=M_1 \oplus M_2$ such that:
      \begin{enumerate}
         \item $M_1 \cong N_1$;
         \item $\phi(M_1) = N_1$;
         \item $\phi(M_2) = N_2$. 
      \end{enumerate}
   \end{Lemma}

   \begin{proof}
      Since $\phi$ is surjective and $N$ is projective, there is a homomorphism $\psi: N \to M$ such that $\phi \psi = \Id_N$. It follows that $M = \Ker(\phi) \oplus \psi(N) = \Ker(\phi) \oplus \psi(N_1) \oplus \psi(N_2)$. Set $M_1=\psi(N_1)$ and $M_2=\Ker(\phi) \oplus \psi(N_2)$ for the desired decomposition.
   \end{proof}

   %%% Definition: Dual Persistence Module %%%
   \begin{Definition}\label{def:dual-pers-mod}
      Let $F:P \to \RMod$ be a totally ordered $P$-persistence module. Let $P^{\text{op}}$ be the poset with opposite ordering denoted by $\leq_{\text{op}}$. Define $F^\vee:P^{\text{op}} \to \RMod$ to be the totally ordered $P^{\text{op}}$-persistence module such that:

      \begin{enumerate}
         \item $F^\vee_x = F_x^*$ for all $x \in P^{\text{op}}$;
         \item $F^\vee_{x,y} = F^*_{y,x}$ for all $x \leq_{\text{op}} y$ in $P^{\text{op}}$.
      \end{enumerate}
   \end{Definition}

   The proof of the following proposition is a generalization of our argument in \cite{gan25}*{\S6} when proving the existence of interval decompositions of pointwise finite dimensional persistence modules over fields. 

   %%% Proposition: Decompose One Projective Constant Module %%%
   \begin{Proposition}\label{prp:total-one-pc-decomp}
      Let $F:P \to \RMod$ be a nonzero p.f.g. totally ordered persistence module. Suppose that the cokernel of every internal morphism $F_{x,y}$ is projective. Then there exists an internal direct sum decomposition $F=\mathbb{I} \oplus G$, where $\mathbb{I}$ is a module of projective constant type.
   \end{Proposition}

   \begin{proof}
      %%% Set Up for Proof %%%
      \begingroup
      Fix $z \in P$ with $F_z \neq 0$. By Lemma~\ref{lem:comp-decom-ima-ker}, there exists an internal direct sum decomposition of $F_z$ that is compatible with the chain of images and the chain of kernels of $F_z$. In particular, there are only finitely many nonzero summands in this decomposition since $F_z$ is a finitely generated projective $R$-module. Hence we can write $F_z=\bigoplus_{i=0}^k N_i$ for some $k \geq 0$, where each $N_i$ is nonzero. Fix $M=N_0$ and let $N= N_1 \oplus \cdots \oplus N_k$ so that $F_z = M \oplus N$. 

      The proof proceeds in three main steps. First, we construct a decomposition of $F$ over indices $x \leq z$ that is compatible with the decomposition at $F_z$; we refer to this as the left-hand decomposition of $F$. Second, we construct an analogous decomposition over indices $x \geq z$ using a dualizing argument, which we call the right-hand decomposition. Finally, we combine these two decompositions to obtain the desired result.	
      \endgroup

      %%% Left-hand Decomposition %%%
      \begingroup
      \textbf{Left-hand Decomposition}

      Let $Q$ be the set of indices $x \leq z$ such that $M$ is a direct summand of $\Ima[x,z]$. Note that $Q$ is nonempty since $z \in Q$. Our goal is to find a sequence of $R$-submodules $\{A_x\}_{x \in Q}$ with $A_x \subseteq F_x$ such that:

      \begin{enumerate}
         \item $A_x \cong M$ for all $x \in Q$; 
         \item $F_{x,z}(A_x)=M$ for all $x \in Q$;
         \item $F_{x,y}(A_x)=A_y$ for all $x<y$ in $Q$.
      \end{enumerate}

      Let $E_x$ be the preimage of $M$ in $F_x$ for all $x \in Q$. Define $\Sigma$ to be the set of all families $(A_x)_{x \in Q}$ with the following conditions:
      
      \begin{enumerate}
         \item[\textbf{(S1)}] $A_x$ is a nontrivial submodule with $A_x \subseteq E_x$ for all $x \in Q$;
         
         \item[\textbf{(S2)}] $F_{x,y}(A_x) \subseteq A_y$ for all $x<y$ in $Q$; 
         
         \item[\textbf{(S3)}] $F_x / A_x$ is projective for all $x \in Q$;
         
         \item[\textbf{(S4)}] $F_y / F_{x,y}(A_x)$ is projective for all $x<y$ in $Q$. \vspace{10pt}
      \end{enumerate} 

      %%% Claim 1: (E_x) is in Sigma %%%
      \textbf{Claim~\ref*{prp:total-one-pc-decomp}.1.} $(E_x)_{x \in Q}$ is a family in $\Sigma$.

         \begin{proof}[Proof of Claim~\ref*{prp:total-one-pc-decomp}.1] 
         Conditions \textbf{(S1)} and \textbf{(S2)} are obvious. 

         Consider the surjection $F_{x,z}:F_x \to \Ima[x,z]$. Since $F_{x,z}(E_x)=M$, there is an induced homomorphism $\overline{F_{x,z}}:F_x/E_x \to \Ima[x,z]/M$. This map is easily verified to be an isomorphism. Hence $F_x/E_x$ is projective since $\Ima[x,z]/M$ is projective, proving condition \textbf{(S3)}. 

         Consider the surjection $F_{y,z}|_{\Ima[x,y]}:\Ima[x,y] \to \Ima[x,z]$. Since $F_{y,z}(F_{x,y}(E_x))=M$, there is an induced homomorphism $\overline{F_{y,z}|_{\Ima[x,y]}}:\Ima[x,y]/F_{x,y}(E_x) \to \Ima[x,z]/M$. This map is easily verified to be an isomorphism. Hence $\Ima[x,y]/F_{x,y}(E_x)$ is projective since $\Ima[x,z]/M$ is projective. We have a short exact sequence
         %%% S.E.S. %%%
         \begin{center}
            \begin{tikzcd}
               0 \arrow [r] 
               & \Frac{\Ima[x,y]}{F_{x,y}(E_x)} \arrow[r, hook]
               & \Frac{F_y}{F_{x,y}(E_x)} \arrow[r, two heads]
               & \Frac{F_y}{\Ima[x,y]} \arrow[r] 
               & 0.
            \end{tikzcd}
         \end{center}

      The sequence splits since $\Coker[x,y]$ is projective, and thus $F_y/F_{x,y}(E_x)$ is projective as it is isomorphic to a direct sum of projective $R$-modules.
      \end{proof}

      %%% Definition of Partial Order on Sigma %%%
      \begingroup
      Define a partial order on $\Sigma$ by $(A_x)_{x \in Q} \leq (B_x)_{x \in Q}$ if $B_x \subseteq A_x$ for all $x \in Q$. Let $\{(A_x[\alpha])\}_\alpha$ be a chain of families in $\Sigma$ indexed by $\alpha$. Define the family $(B_x)_{x \in Q}$ by 
      
      \begin{center}
         $B_x = \bigcap\limits_{\alpha} A_x[\alpha]$.
      \end{center}
      \endgroup
      
      %%% Claim 2: (B_x) in Sigma %%%
      \textbf{Claim~\ref*{prp:total-one-pc-decomp}.2.} $(B_x)_{x \in Q}$ is in $\Sigma$.
      
      \begin{proof}[Proof of Claim~\ref*{prp:total-one-pc-decomp}.2.]
         At each fixed index $x$, $\{A_x[\alpha]\}_{\alpha}$ is a chain of projective $R$-submodules of $F_x$ with $F_x/A_x[\alpha]$ projective for every $\alpha$. Lemma~\ref{lem:fin-chain} guarantees there are only finitely many submodules $A_x[\alpha]$ in this chain. Therefore $B_x$ will be equal to the minimal submodule in the chain $\{A_x[\alpha]\}_{\alpha}$. This guarantees conditions $\textbf{(S1)}$, \textbf{(S3)} and \textbf{(S4)} are satisfied.

         Finally, since $B_x \subseteq A_x[\alpha]$ for all $\alpha$, $F_{x,y}(B_x) \subseteq F_{x,y}(A_x[\alpha]) \subseteq A_y[\alpha]$ for all $\alpha$. Hence $F_{x,y}(B_x) \subseteq \bigcap\limits_{\alpha} A_y[\alpha]= B_y$, satisfying condition \textbf{(S2)}.
      \end{proof}

      %%% (B_x) is an Upper Bound %%%
      \begingroup
      Now $(B_x)_{x \in Q}$ is in $\Sigma$ and $(A_x[\alpha])_{x \in Q} \leq (B_x)_{x \in Q}$ for every $\alpha$ since $B_x \subseteq A_x[\alpha]$ for every $x \in Q$. This shows $(B_x)_{x \in Q}$ is an upper bound of the chain $\{(A_x[\alpha])\}_\alpha$. Zorn's Lemma implies the existence of a maximal element, say $(A_x)_{x \in Q}$, in $\Sigma$.

      Define $A'_x$ to be the minimal image of $A_r$ in $A_x$ for all $r<x$ in $Q$, where the minimal image here is with regard to subset containment.
      \endgroup

      %%% Claim 3: The Minimal Image A'_x Exists %%%
      \textbf{Claim~\ref*{prp:total-one-pc-decomp}.3.} The minimal image $A'_x$ exists in $A_x$.

      \begin{proof}[Proof of Claim~\ref*{prp:total-one-pc-decomp}.3.]
         The set of images $\{F_{r,x}(A_r): r \leq x\}$ form a chain of projective $R$-submodules of $F_x$, where $F_x/F_{r,x}(A_r)$ is projective for every $r \leq x$ by condition \textbf{(S4)}. Lemma~\ref{lem:fin-chain} guarantees there are only finitely many submodules in this chain. Therefore a minimal image $A'_x$ exists.
      \end{proof}

      %%% Claim 4: A'x maps surjectively to A'y %%%
      \textbf{Claim~\ref*{prp:total-one-pc-decomp}.4.} $A'_x$ maps surjectively to $A'_y$ whenever $x \leq y$.

      \begin{proof}[Proof of Claim~\ref*{prp:total-one-pc-decomp}.4.]
         By definition, $A'_x=F_{r,x}(A_r)$ for some $r \leq x$ and $A'_y=F_{w,y}(A_w)$ for some $w \leq y$. We have two cases to check.

         Suppose $r \leq w$. Then $F_{r,w}(A_r) \subseteq A_w$ implies $F_{r,y}(A_r) \subseteq F_{w,y}(A_w) = A'_y$. Minimality of $A'_y$ requires $F_{r,y}(A_r) = A'_y$. Hence
         \begin{center}
            $A'_y= F_{r,y}(A_r) = F_{x,y}F_{r,x}(A_r) = F_{x,y}(A'_x)$. 
         \end{center}

         Now suppose $w < r$. Then $F_{w,r}(A_w) \subseteq A_r$ implies $F_{w,x}(A_w) \subseteq F_{r,x}(A_r) = A'_x$. Minimality of $A'_x$ requires $F_{w,x}(A_w) = A'_x$. Hence 
         \begin{center}
            $F_{x,y}(A'_x) = F_{x,y}F_{w,x}(A_w) = F_{w,y}(A_w)=A'_y$.
         \end{center} 
      \end{proof}

      %%% Claim 5: A'x is in Sigma %%%
      \textbf{Claim~\ref*{prp:total-one-pc-decomp}.5.} $(A'_x)_{x \in Q}$ is in $\Sigma$.

      \begin{proof}[Proof of Claim~\ref*{prp:total-one-pc-decomp}.5.]
         By definition, $A'_x=F_{r,x}(A_r)$ for some $r \leq x$. This implies conditions \textbf{(S1)} and \textbf{(S3)} are immediately satisfied. By Claim~\ref*{prp:total-one-pc-decomp}.4., we have $F_{x,y}(A'_x) = A'_y$. Thus condition \textbf{(S2)} is satisfied, and moreover $F_y/F_{x,y}(A'_x) = F_y/A'_y$ is projective which satisfies \textbf{(S4)}. 
      \end{proof}

      %%% Subproof: Ax maps surjectively to Ay, and A'x=Ax %%%
      Now $(A'_x)_{x \in Q}$ is in $\Sigma$ with $(A_x)_{x \in Q} \leq (A'_x)_{x \in Q}$. Therefore maximality of $(A_x)_{x \in Q}$ implies $(A_x)_{x \in Q}=(A'_x)_{x \in Q}$. Since it was shown in Claim~\ref*{prp:total-one-pc-decomp}.4 that $A'_x$ maps to $A'_y$ surjectively for all $x \leq y$, we must have $A_x$ maps to $A_y$ surjectively for all $x \leq y$.

      %%% Definition of Cx %%%
      \begingroup
      Fix $w \in Q$. The map $F_{w,z}|_{A_w}:A_w \to M$ is a surjection. We can apply Lemma~\ref{lem:decomp-surj-map} to decompose $A_w = M_w \oplus K_w$ so that $M_w \cong M, F_{w,z}(M_w)=M$, and $F_{w,z}(K_w)=0$. 
      
      Define the family $(C_x)_{x \in Q}$ by 
      
      \begin{center}
         $C_x= \begin{cases} 
            \Inv{F_{x,w}}(M_w) \cap A_x & \text{if } x < w, \\
            M_w & \text{if } x=w, \\
            A_x & \text{if } x > w, \\ 
         \end{cases}$
      \end{center}

      where $\Inv{F_{x,w}}(M_w)$ denotes the preimage of $M_w$ in $F_x$.
      \endgroup

      %%% Claim 6: (C_x) is a family in Sigma %%%
      \textbf{Claim~\ref*{prp:total-one-pc-decomp}.6.} $(C_x)_{x\in Q}$ is a family in $\Sigma$.

      \begin{proof}[Proof of Claim~\ref*{prp:total-one-pc-decomp}.6.]
         Condition \textbf{(S1)} is obvious for $x \geq w$. When $x<w$, surjectivity of $A_x$ onto $A_w$ guarantees that $\Inv{F_{x,w}}(M_w) \cap A_x$ will be a nonzero $R$-submodule contained in $E_x$. Condition \textbf{(S2)} is obvious for all indices. It remains to carefully justify conditions \textbf{(S3)} and \textbf{(S4)}. 
         
         Condition \textbf{(S3)} is immediate when $x > w$. For $x=w$, we have a short exact sequence 
         %%% S.E.S. %%%
         \begin{center}
            \begin{tikzcd}
               0 \arrow [r] 
               & \Frac{A_w}{M_w} \arrow[r, hook]
               & \Frac{F_w}{M_w} \arrow[r, two heads]
               & \Frac{F_w}{A_w} \arrow[r] 
               & 0.
            \end{tikzcd}
         \end{center}
         
         The sequence splits, and thus $F_w/M_w$ is projective as it is isomorphic to a direct sum of projective $R$-modules.

         Now suppose $x < w$. Consider the surjection $F_{x,w}|_{A_x}:A_x \to A_w$. There is an induced homomorphism $\overline{F_{x,w}|_{A_x}}:A_x/(\Inv{F_{x,w}}(M_w) \cap A_x) \to A_w/M_w$. This map is easily verified to be an isomorphism. Hence $A_x/(\Inv{F_{x,w}}(M_w) \cap A_x)$ is projective since $A_w/M_w$ is projective. Moreover, we have a short exact sequence

         %%% S.E.S. %%%
         \begin{center}
            \begin{tikzcd}
               0 \arrow [r] 
               & \Frac{A_x}{\Inv{F_{x,w}}(M_w) \cap A_x} \arrow[r, hook]
               & \Frac{F_x}{\Inv{F_{x,w}}(M_w) \cap A_x} \arrow[r, two heads]
               & \Frac{F_x}{A_x} \arrow[r] 
               & 0.
            \end{tikzcd}
         \end{center}

         The sequence splits, and thus $F_x/(\Inv{F_{x,w}}(M_w) \cap A_x)$ is projective as it is isomorphic to a direct sum of projective $R$-modules. This proves condition \textbf{(S3)}. 

         When $w \leq x < y$ or $x < w \leq y$, condition \textbf{(S4)} is obvious. Suppose $x < y < w$. We claim $F_{x,y}(\Inv{F_{x,w}}(M_w) \cap A_x) = \Inv{F_{y,w}}(M_w) \cap A_y$. 
         
         Indeed, the forward inclusion is straightforward. To see the reverse inclusion, suppose $a_y \in \Inv{F_{y,w}}(M_w) \cap A_y$. By surjectivity, there exists $a_x \in A_x$ such that $F_{x,y}(a_x)=a_y$. This implies $F_{x,w}(a_x)=F_{y,w}(a_y) \in M_w$. Hence $a_x \in \Inv{F_{x,w}}(M_w) \cap A_x$, implying $a_y \in F_{x,y}(\Inv{F_{x,w}}(M_w) \cap A_x)$. 

         The equality above means $F_y / F_{x,y}(\Inv{F_{x,w}}(M_w) \cap A_x) = F_y /(\Inv{F_{y,w}}(M_w) \cap A_y)$, which was shown to be projective by condition \textbf{(S3)}. 
      \end{proof}

      %%% Maximality of C_x %%%
      Now $(C_x)_{x\in Q}$ is in $\Sigma$ with $(A_x)_{x \in Q} \leq (C_x)_{x \in Q}$. Therefore maximality of $(A_x)_{x \in Q}$ implies $(C_x)_{x \in Q} = (A_x)_{x \in Q}$, and in particular that $A_w=M_w$. For every $x \geq w$, surjectivity of $A_w$ onto $A_x$ combined with the fact that $A_w \cong M$ implies $A_w$ maps bijectively to $A_x$ so that $A_x \cong M$ for all $w \leq x \leq z$ in $Q$. Since $w$ was arbitrary, the maximal element $(A_x)_{x \in Q}$ must be a family satisfying:
      \begin{enumerate}
         \item $A_x \cong M$ for all $x \in Q$;
         \item $F_{x,z}(A_x)=M$ for all $x \in Q$;
         \item $F_{x,y}(A_x)=A_y$ for all $x<y$ in $Q$.
      \end{enumerate}

      %%% Finish Left-hand Decomposition %%%	
      Define the persistence submodules $\mathbb{I}^-$ and $G^-$ of $F|_{\{x \leq z\}}$ as follows:
      \begin{align*}
         \mathbb{I}^-_x = 
         \begin{cases} 
            A_x & \text{if } x \in Q, \\
            0 & \text{if } x \notin Q; \\
         \end{cases} 
         \hspace{15pt} 
         & G^-_x = \Inv{F_{x,z}}(N) \text{ for all } x \leq z.
      \end{align*}

      Then $F|_{\{x \leq z\}} = \mathbb{I}^- \oplus G^-$ where $\mathbb{I}^-$ is of projective constant type.
      \endgroup

      %%% Right-hand Decomposition %%%
      \begingroup
      \textbf{Right-hand Decomposition}

      %%% Set Up %%%
      \begingroup
      Recall that $F_z=\bigoplus_{i=0}^k N_i$ where $N_0=M$. Let $S$ be the set of indices $x \geq z$ such that $N_0$ is not a direct summand of $\Ker[z,x]$. Note that $S$ is nonempty since $z \in S$.   
      
      We will now focus on the persistence module $F|_{S}$. We will abuse notation and simply write $F$ instead of $F|_S$ below as it is understood from context, and view $F$ as an $S$-persistence module. 

      Consider the $S^{\text{op}}$-persistence module $F^\vee$ where $\leq_{\text{op}}$ denotes the opposite ordering in $S^{\text{op}}$ (see Definition~\ref{def:dual-pers-mod}). Since $\Coker(F_{x,y})$ is projective for every $x \leq y$ in $S$, $\Coker(F^\vee_{y,x})=\Coker(F^*_{x,y})$ is projective for every $y \leq_{\text{op}} x$ in $S^{\text{op}}$ by Lemma~\ref{lem:coker-dual-map-proj}.   

      By Definition-Lemma~\ref{deflem:dual-basis-proj}, there exists a pair of dual bases $\{(a_{i,j}, f_{i,j}) : 1 \leq j \leq \iota_i\}$ for each submodule $N_i$ and for some finite number of elements $\iota_i$. Using Lemma~\ref{lem:dual-dir-sum}, we can write $F^\vee_z=\bigoplus_{i=0}^k N'_i$ where each $N'_i$ is generated by the set of linear functionals $\{\bar{f}_{i,j} : 1 \leq j \leq \iota_i\}$. Recall here that $\bar{f}_{i,j}:F \to R$ is the extension of the linear functional $f_{i,j}:N_i \to R$ by setting $\bar{f}_{i,j}|_{N_r}=0$ for all $r \neq i$.  
      \endgroup

      %%% Claim 7: Decomposition is Compatible with Chain of Images %%%
      \textbf{Claim~\ref*{prp:total-one-pc-decomp}.7.} The decomposition $F^\vee_z=\bigoplus_{i=0}^k N'_i$ is compatible with the chain of images of $F^\vee_z$. 

      \begin{proof}[Proof of Claim~\ref*{prp:total-one-pc-decomp}.7.] 	
         To see this, let $x > z$. If $\Ker[z,x]=0$, the claim is immediate. Suppose $\Ker[z,x] \neq 0$. After possible reindexing of the submodules $N_i$ (excluding $i=0$), we can write $\Ker[z,x] = N_1 \oplus \cdots \oplus N_d$ for some $d \leq k$. 

         By Lemma~\ref{lem:dual-ann-ker-ima}, $\Ima(F^\vee_{x,z})=\Ima(F^*_{z,x}) = \Ker[z,x]^\perp$ for $x \leq_{\text{op}} z$. But it is easy to see that $\Ker[z,x]^\perp=N'_0 \oplus N'_{d+1} \oplus \cdots \oplus N'_{k}$, which implies the decomposition is compatible (see Definition~\ref{def:compatible-chain}).  

         Indeed, the right-hand side is obviously in the left-hand side. If $f \in \Ker[z,x]^\perp$, then we write $f=f_0+f_1+\cdots+f_k$ where each $f_i \in N'_i$. If $f_i$ is a nonzero linear functional on $N_i$ for some $1 \leq i \leq d$, then there exists $a \in N_i$ such that $f_i(a)$ is nonzero. But then $f(a)=f_i(a)$ is nonzero, a contradiction.
      \end{proof}

      %%% Apply Left-hand Decomposition %%%
      \begingroup
      Since the decomposition $F^\vee_z=\bigoplus_{i=0}^k N'_i$ is compatible with the chain of images of $F^\vee_z$, we can apply the results of the left-hand decomposition to $F^\vee$ so that $F^\vee = \mathcal{I} \oplus \mathcal{G}$ with the following properties:

      \begin{enumerate}
         \item $\mathcal{I}_z=N'_0$ and $\mathcal{G}_z = N'_1 \oplus \cdots \oplus N'_k$;
         
         \item $F^\vee_{x,y}(\mathcal{I}_x) = F_{y,x}^*(\mathcal{I}_x) = \mathcal{I}_y$ for all $x \leq_\text{op} y$ in $S^{\text{op}}$;
         
         \item $F^\vee_{x,y}(\mathcal{G}_x) = F_{y,x}^*(\mathcal{G}_x) \subseteq \mathcal{G}_y$ for all $x \leq_\text{op} y$ in $S^{\text{op}}$;
         
         \item $\mathcal{I}_x \cong N'_0 \cong N_0^*$ for all $x \in S^{\text{op}}$. 
      \end{enumerate}
      \endgroup

      %%% Dualize Again %%%
      \begingroup
      Dualizing again, consider the $S$-persistence module $F^{\vee \vee}$ where $\leq$ denotes the original total ordering on $S$. Observe by Definition~\ref{def:dual-pers-mod} that 
      
      \begin{enumerate}
         \item $F^{\vee \vee}_x=F^{**}_x$ for all $x \in S$; 
         \item $F^{\vee\vee}_{x,y} = F^{**}_{x,y}$ for all $x \leq y$ in $S$. 
      \end{enumerate}
      
      Since $\Coker(F^\vee_{y,x})$ is projective for every $y \leq_{\text{op}} x$ in $S^{\text{op}}$, 
      
      \begin{center}
         $\Coker(F^{\vee\vee}_{x,y})=\Coker((F^\vee_{y,x})^*)=\Coker(F^{**}_{x,y})$
      \end{center}

      is projective for every $x \leq y$ in $S$ by Lemma~\ref{lem:coker-dual-map-proj}. We will now use the decomposition $F^\vee = \mathcal{I} \oplus \mathcal{G}$ to construct a right-hand decomposition of $F^{\vee \vee}$.

      Observe that $\bigcup_{i=0}^k \{(a_{i,j}, \bar{f}_{i,j}) : 1 \leq j \leq \iota_i\}$ is a pair of dual bases of $F_z$ by Lemma~\ref{lem:dual-dir-sum}. Thus $\bigcup_{i=0}^k \{(\bar{f}_{i,j}, \hat{a}_{i,j}) : 1 \leq j \leq \iota_i\}$ is a pair of dual bases of $F^\vee_z$ by Lemma~\ref{lem:dual-basis-of-dual}. In particular, $\mathcal{I}_z$ is generated by the set $\{\bar{f}_{0,j} : 1 \leq j \leq \iota_0\}$ and $\mathcal{G}_z$ is generated by the set $\bigcup_{i=1}^k \{\bar{f}_{i,j} : 1 \leq j \leq \iota_i\}$. 

      Applying Lemma~\ref{lem:dual-dir-sum}, we can write $F^{\vee\vee}_z = \mathcal{I}'_z \oplus \mathcal{G}'_z$ where $\mathcal{I}'_z$ is generated by the set $\{\hat{a}_{0,j} : 1 \leq j \leq \iota_0\}$ and $\mathcal{G}'_z$ is generated by the set $\bigcup_{i=1}^{k} \{\hat{a}_{i,j} : 1 \leq j \leq \iota_i\}$. 
      
      Similarly, for any $x > z$ in $S$ with $F^\vee_x = \mathcal{I}_x \oplus \mathcal{G}_x$, we can write $F^{\vee\vee}_x=\mathcal{I}'_x \oplus \mathcal{G}'_x$ as in Lemma~\ref{lem:dual-dir-sum}. 
      
      Since $F^\vee_{y,x}(\mathcal{I}_y) \subseteq \mathcal{I}_x$ and $F^\vee_{y,x}(\mathcal{G}_y) \subseteq \mathcal{G}_x$ for any $y \leq_{\text{op}} x$ in $S^{\text{op}}$, Lemma~\ref{lem:comp-maps-dual} implies $F^{\vee\vee}_{x,y}(\mathcal{I}'_x) \subseteq \mathcal{I}'_y$ and $F^{\vee\vee}_{x,y}(\mathcal{G}'_x) \subseteq \mathcal{G}'_y$ for any $x \leq y$ in $S$. In particular, Lemma~\ref{lem:comp-maps-dual} guarantees that $F^{\vee\vee}_{x,y}$ maps $\mathcal{I}'_x$ to $\mathcal{I}'_y$ bijectively for all $x \leq y$ in $S$ since $F^\vee_{y,x}$ maps $\mathcal{I}_y$ to $\mathcal{I}_x$ bijectively for all $y \leq_{\text{op}} x$ in $S^{\text{op}}$.

      For all $x \in S$, let $\varepsilon_x:F_x \to F^{\vee\vee}_x$ be the natural isomorphism mapping $a \mapsto \hat{a}$ for any $a \in F_x$ (see Definition-Lemma~\ref{deflem:nat-double-dual-isom}). For any $x \leq y$ in $S$, we have the following commuting diagram of morphisms. 

      \begin{center}
         \begin{tikzcd}
            F_z \arrow[r, "F_{z,x}"] \arrow[d, "\cong"] \arrow[d, "\varepsilon_z", swap] \arrow[rr, "F_{z,y}", bend left=20]
            & F_x \arrow[r, "F_{x,y}"] \arrow[d, "\cong"] \arrow[d, "\varepsilon_x", swap] 	
            & F_y \arrow[d, "\cong"] \arrow[d, "\varepsilon_y", swap] \\

            F^{\vee\vee}_z = \mathcal{I}'_z \oplus \mathcal{G}'_z \arrow[r, "F^{\vee\vee}_{z,x}"] \arrow[rr, "F^{\vee\vee}_{z,y}", swap, bend right=20]
            & F^{\vee\vee}_x = \mathcal{I}'_x \oplus \mathcal{G}'_x \arrow[r, "F^{\vee\vee}_{x,y}"] 
            & F^{\vee\vee}_y = \mathcal{I}'_y \oplus \mathcal{G}'_y \\
         \end{tikzcd}
      \end{center}
      \endgroup

      %%% Final Right-hand Decomposition %%%
      \begingroup
      We can use this diagram to pull back the decomposition of $F^{\vee\vee}$ to a decomposition of $F$ for indices $x \in S$. Define the persistence submodules $\mathbb{I}^+$ and $G^+$ of $F|_{\{x \geq z\}}$ as follows:
      \begin{align*}
         \mathbb{I}^+_x = 
         \begin{cases} 
            \Inv{\varepsilon_x}(\mathcal{I}'_x) & \text{if } x \in S, \\
            0 & \text{if } x \notin S; \\
         \end{cases} 
         \hspace{15pt} 
         & G^+_x = 
         \begin{cases} 
            \Inv{\varepsilon_x}(\mathcal{G}'_x) & \text{if } x \in S, \\
            F_x & \text{if } x \notin S. \\
         \end{cases}
      \end{align*}
      
      Then $F|_{\{x \geq z\}} = \mathbb{I}^+ \oplus G^+$ where $\mathbb{I}^+$ is of projective constant type. In particular, we have $\mathbb{I}^+_z=M$ and $G^+_z=N$ since $\varepsilon_z(a_{i,j})=\hat{a}_{i,j}$ for all such indices $i,j$.
      \endgroup
      \endgroup

      %%% Overall Decomposition %%%
      \begingroup
      \textbf{Overall Decomposition}

      The left-hand and right-hand decompositions constructed above are compatible at the index $z$ with $\mathbb{I}^-_z=\mathbb{I}^+_z=M$ and $G^-_z=G^+_z=N$. Define the $P$-persistence submodule $\mathbb{I}$ so that $\mathbb{I}_x = \mathbb{I}^-_x$ if $x \leq z$ and $\mathbb{I}_x = \mathbb{I}^+_x$ if $x > z$. Likewise, define the $P$-persistence submodule $G$ so that $G_x = G^-_x$ if $x \leq z$ and $G_x = G^+_x$ if $x > z$. Then $F=\mathbb{I} \oplus G$ where $\mathbb{I}$ is a persistence submodule of projective constant type.
      \endgroup
   \end{proof}

   %%% Lemma: Submodule Projective Cokernel %%%
   \begin{Lemma}\label{lem:proj-coker-submod}
      Let $F:P \to \RMod$ be a totally ordered persistence module. Suppose that the cokernel of every internal morphism $F_{x,y}$ is projective. If $F=G \oplus H$, then the cokernels of every internal morphism $G_{x,y}$ and $H_{x,y}$ are also projective. 
   \end{Lemma}

   \begin{proof}
      For any $x \leq y$, it is easy to verify that $\Ima[x,y]=F_{x,y}(G_x)\oplus F_{x,y}(H_x)$. Observe that 

      \begin{center}
         $\Coker[x,y] = \Frac{F_y}{\Ima[x,y]} = \Frac{G_y \oplus H_y}{F_{x,y}(G_x)\oplus F_{x,y}(H_x)} \cong \bigg( \Frac{G_y}{F_{x,y}(G_x)} \bigg) \oplus \bigg( \Frac{H_y}{F_{x,y}(H_x)}$ \bigg). 
      \end{center}

      Since $\Coker[x,y]$ is projective, both $G_y/F_{x,y}(G_x)$ and $H_y/F_{x,y}(H_x)$ are projective. But $G_{x,y}(G_x)=F_{x,y}(G_x)$ and $H_{x,y}(H_x)=F_{x,y}(H_x)$. This implies that $\Coker(G_{x,y})$ and $\Coker(H_{x,y})$ are both projective for any $x \leq y$. 
   \end{proof}

   The above lemma guarantees that the desired cokernel property descends to direct summands. With this, we are finally ready to state the main theorem of this section, namely Theorem~\ref{thm:main-pc-total}.  

   %%% Theorem: Complete Totally Ordered Projective Constant Decomposition %%%
   \begin{Theorem}\label{thm:total-complete-pc-decomp}
      Let $F:P \to \RMod$ be a nonzero p.f.g. totally ordered persistence module. Suppose that the cokernel of every internal morphism $F_{x,y}$ is projective. Then $F$ admits a projective constant decomposition.
   \end{Theorem}

   \begin{proof}
      Immediate from Lemma~\ref{lem:decomp-criteria}, Proposition~\ref{prp:total-one-pc-decomp} and Lemma~\ref{lem:proj-coker-submod} when letting $\mathcal{E}$ in Lemma~\ref{lem:decomp-criteria} be the set of submodules of $F$ of projective constant type. 
   \end{proof}

   We deduce Corollary~\ref{cor:main-pid-total} in the following result. 

   %%% Corollary: PID or Local Ring %%% 
   \begin{Corollary}\label{cor:total-pid-decomp}
      Let $R$ be a commutative Noetherian ring with identity such that every finitely generated projective $R$-module is free. Let $F:P \to \RMod$ be a nonzero p.f.g. totally ordered persistence module. Suppose that the cokernel of every internal morphism $F_{x,y}$ is projective. Then $F$ is interval decomposable.
   \end{Corollary}

   \begin{proof}
      Immediate from Theorem~\ref{thm:total-complete-pc-decomp} and Lemma~\ref{lem:free-interval-decomp}.
   \end{proof}

   The result of Corollary~\ref{cor:total-pid-decomp} when $R$ is a PID was also proved by Luo and Henselman-Petrusek \cites{luo23, luo25}.

\section{Zigzag Persistence Modules: Finite Extrema}\label{sec:finite-zigzag}

   Throughout this section, we assume that $R$ is a commutative Noetherian ring with identity, unless otherwise specified. Building on the decomposition results for totally ordered persistence modules established in Section~\ref{sec:total-order}, we now turn to our main decomposition theorem for zigzag persistence modules with finite extrema. The following result is the statement of Theorem~\ref{thm:main-pc-zigzag} when $\Gamma = \{1,2, \ldots, n\}$ for $n \geq 2$. The proof follows a similar argument as that of Theorem~\ref{thm:quiver-pc-decomp}. 

   %%% Theorem: Finite Zigzag PC Decomp %%%
   \begin{Theorem}\label{thm:fin-zigzag-pc-decomp}
      Let $F:P \to \RMod$ be a nonzero p.f.g. zigzag persistence module with extrema $\{z_1,z_2, \ldots, z_n\}$ for $n \geq 2$ satisfying the PCC. Then $F$ admits a projective constant decomposition.
   \end{Theorem}

   \begin{proof}
      %%% Case for N=1 and N=2 %%%
      \begingroup
         We will proceed by induction on $n$. For the base case of $n=2$, $F$ is a nonzero p.f.g. totally ordered persistence module such that the cokernel of every internal morphism $F_{x,y}$ is projective (see Lemma~\ref{lem:pcc-more-facts}). Theorem~\ref{thm:total-complete-pc-decomp} guarantees $F$ admits a projective constant decomposition.
      \endgroup

      %%% Reduce to N=3 %%%
      \begingroup
         Let $n \geq 3$ and assume the inductive hypothesis holds up to $n-1$. Let $\sqsubseteq$ be the associated total order. As in the proof of Theorem~\ref{thm:quiver-pc-decomp}, we can reduce to the case of showing that $F$ admits a projective constant decomposition when $n=3$ and $F$ has a peak at $z_2$. We will not rewrite all the details of this claim here as they are completely analogous to the proof of Theorem~\ref{thm:quiver-pc-decomp}.
      \endgroup
      
      %%% N=3 Case #1 %%%
      \begingroup
         \textbf{Case 1.} Consider the $n=3$ case with the following diagram.
         
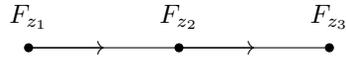
\begin{center}
               \begin{tikzpicture}
                  \draw[black, -] (-2,0) -- (2,0); %Number line
                  
                  %Points		
                  \filldraw [black] (-2,0) circle (1.5pt);
                  \node[label={90:$F_{z_1}$}] at (-2,0) {};
                  
                  \filldraw [black] (0,0) circle (1.5pt);
                  \node[label={90:$F_{z_2}$}] at (0,0) {};
                  
                  \filldraw [black] (2,0) circle (1.5pt);
                  \node[label={90:$F_{z_3}$}] at (2,0) {};
                  
                  %Arrows
                  \draw[black, ->] (-2,0) -- (-1,0); 
                  \draw[black, ->] (0,0) -- (1,0); 
               \end{tikzpicture}
         \end{center}

         This is a nonzero p.f.g. totally ordered persistence module such that the cokernel of every internal morphism $F_{x,y}$ is projective (see Lemma~\ref{lem:pcc-more-facts}). Theorem~\ref{thm:total-complete-pc-decomp} guarantees that $F$ admits a projective constant decomposition. When the morphisms are reversed with $z_3 \leq z_2 \leq z_1$, the same argument holds.
      \endgroup

      %%% N=3 Case #2 %%%
      \begingroup 
      \textbf{Case 2.} Consider the $n=3$ case with the following diagram.
      
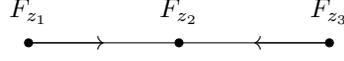
\begin{center}
         \begin{tikzpicture}
            \draw[black, -] (-2,0) -- (2,0); %Number line
            
            %Points		
            \filldraw [black] (-2,0) circle (1.5pt);
            \node[label={90:$F_{z_1}$}] at (-2,0) {};
            
            \filldraw [black] (0,0) circle (1.5pt);
            \node[label={90:$F_{z_2}$}] at (0,0) {};
            
            \filldraw [black] (2,0) circle (1.5pt);
            \node[label={90:$F_{z_3}$}] at (2,0) {};
            
            %Arrows
            \draw[black, ->] (-2,0) -- (-1,0); 
            \draw[black, <-] (1,0) -- (2,0); 
         \end{tikzpicture}
      \end{center}

      Note all morphisms are injective since $z_2$ is a peak, and $F_{z_2} \neq 0$.
      
      Observe that $F|_{[z_1,z_2]}$ and $F|_{[z_2,z_3]}$ are both totally ordered persistence modules such that the cokernel of every internal morphism is projective. Lemma~\ref{lem:fin-chain-ima-ker} implies there is a finite chain of images $L=\{\Ima[x,z_2] : z_1 \sqsubseteq x \sqsubseteq z_2 \}$ of $F_{z_2}$ and another finite chain of images $R=\{\Ima[x,z_2] : z_2 \sqsubseteq x \sqsubseteq z_3 \}$ of $F_{z_2}$.

      Let $l$ and $r$ be the number of unique images of $L$ and $R$, respectively. Choose indices $x_i \sqsubseteq z_2$ for $1 \leq i \leq l$ so that each $\Ima[x_i,z_2]$ is one of the unique images in the chain $L$. Likewise, choose indices $y_j \sqsupseteq z_2$ for indices $1 \leq j \leq r$ so that each $\Ima[y_j,z_2]$ is one of the unique images in the chain $R$. 
      
      Let $S = \{x_i\}_{i=1}^l \cup \{y_j\}_{j=1}^r \cup \{z_2\}$. Then $F|_S$ is a nonzero $A_n$-persistence module satisfying the PCC, and therefore admits a projective constant decomposition by Theorem~\ref{thm:quiver-pc-decomp}. This decomposition at the index $z_2$ looks like a finite internal direct sum of projective $R$-modules $F_{z_2}=\bigoplus_{k=1}^m M_k$ that is compatible with both chains of images $L$ and $R$.
      
      Fix $M=M_1$ and let $N=M_2 \oplus \cdots \oplus M_m$. Then $F_{z_2} = M \oplus N$. Let $Q$ be the set of indices $x \in P$ such that $M$ is a direct summand of $\Ima[x,z_2]$. 
      
      Define $\mathbb{I}_x$ to be the preimage of $M$ in $F_x$ for every $x \in Q$, and $\mathbb{I}_x=0$ for every $x \notin Q$. For all indices $x \in P$, define $G_x$ to be the preimage of $N$ in $F_x$. We claim that $F_x = \mathbb{I}_x \oplus G_x$ for every $x \in P$, and also that $\mathbb{I}_x \cong M$ for every $x \in Q$. 
      
      Indeed, the claim is obvious for $x \notin Q$ since $F_{x,z_2}$ is injective. For $x \in Q$, consider $\Ima[x,z_2]$ in $F_{z_2}$. Then $\Ima[x,z_2] = \bigoplus_{k \in K}M_k$ for some subset of indices $K \subseteq \{1, \ldots, m\}$. Since $F_{x,z_2}$ is injective, it is easy to see how we pull back this decomposition to yield $F_{x}=\mathbb{I}_x \oplus G_x$, and moreover that $\mathbb{I}_x \cong M$. 
      
      Hence $F = \mathbb{I} \oplus G$ where $\mathbb{I}$ is a persistence submodule of projective constant type. If $G=0$, we are done. If $G \neq 0$, then $G$ has a peak at $z_2$ by Lemma~\ref{lem:peak-facts} and satisfies the PCC by Lemma~\ref{lem:pcc-more-facts}. We can repeat the above process on $G$ to decompose another direct summand of projective constant type. After $m$ iterations, we obtain a complete projective constant decomposition of $F$.
      \endgroup

      %%% N=3 Case #3 %%%
      \begingroup 
      \textbf{Case 3.} Consider the $n=3$ case with the following diagram.	
      
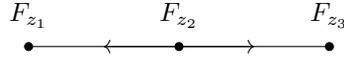
\begin{center}
         \begin{tikzpicture}
            \draw[black, -] (-2,0) -- (2,0); %Number line
            
            %Points		
            \filldraw [black] (-2,0) circle (1.5pt);
            \node[label={90:$F_{z_1}$}] at (-2,0) {};
            
            \filldraw [black] (0,0) circle (1.5pt);
            \node[label={90:$F_{z_2}$}] at (0,0) {};
            
            \filldraw [black] (2,0) circle (1.5pt);
            \node[label={90:$F_{z_3}$}] at (2,0) {};
            
            %Arrows
            \draw[black, ->] (0,0) -- (1,0); 
            \draw[black, <-] (-1,0) -- (0,0); 
         \end{tikzpicture}
      \end{center}

      Note all morphisms are surjective since $z_2$ is a peak, and $F_{z_2} \neq 0$. 
      
      Observe that $F|_{[z_1,z_2]}$ and $F|_{[z_2,z_3]}$ are both totally ordered persistence modules such that the cokernel of every internal morphism is projective. Lemma~\ref{lem:fin-chain-ima-ker} implies there is a finite chain of kernels $L=\{\Ker[z_2,x] : z_1 \sqsubseteq x \sqsubseteq z_2 \}$ of $F_{z_2}$ and another finite chain of kernels $R=\{\Ker[z_2,x] : z_2 \sqsubseteq x \sqsubseteq z_3 \}$.
      
      By an argument analogous to Case 2, $F_{z_2}$ admits an internal direct sum decomposition $F_{z_2}=\bigoplus_{k=1}^m M_k$ of projective $R$-modules that is compatible with both chains of kernels $L$ and $R$.

      Fix $M=M_1$ and $N = M_2 \oplus \cdots \oplus M_m$. Then $F_{z_2}=M \oplus N$. Let $Q$ be the set of indices $x \in P$ such that $M$ is not a direct summand of $\Ker[z_2,x]$. 
      
      Define $\mathbb{I}_x=F_{z_2,x}(M)$ for every $x \in Q$, and $\mathbb{I}_x=0$ for every $x \notin Q$. For all indices $x \in P$, define $G_x=F_{z_2,x}(N)$. We claim that $F_x=\mathbb{I}_x \oplus G_x$ for every $x \in P$, and also that $\mathbb{I}_x \cong M$ for every $x \in Q$. 
      
      Indeed, the claim is obvious for $x \notin Q$ since $F_{z_2,x}$ is surjective. For $x \in Q$, clearly $\mathbb{I}_x + G_x$ generates all of $F_x$ by surjectivity of $F_{z_2,x}$. To see the sum is direct, suppose $a \in \mathbb{I}_x \cap G_x$. Then there exist $b \in M, c \in N$ such that $F_{z_2,x}(b)=F_{z_2,x}(c)=a$. This implies $b - c \in \Ker[z_2,x]$, where $\Ker[z_2,x] \subseteq N$ since $M$ is not a direct summand of $\Ker[z_2,x]$. Hence $b \in M \cap N \implies b=0 \implies a=0$. Moreover, $F_{z_2,x}|_{M}: M \to \mathbb{I}_x$ is clearly a bijection.

      Hence $F = \mathbb{I} \oplus G$ where $\mathbb{I}$ is a persistence submodule of projective constant type. If $G=0$, we are done. If $G \neq 0$, then $G$ has a peak at $z_2$ by Lemma~\ref{lem:peak-facts} and satisfies the PCC by Lemma~\ref{lem:pcc-more-facts}. We can repeat the above process on $G$ to decompose another direct summand of projective constant type. After $m$ iterations, we obtain a complete projective constant decomposition of $F$.
      \endgroup
   \end{proof}

   We deduce Corollary~\ref{cor:main-pid-zigzag} when $\Gamma = \{1,2, \ldots, n\}$ for $n \geq 2$ in the following result.

   %%% Corollary: PID or Local Ring %%%
   \begin{Corollary}\label{cor:fin-zigzag-int-decomp}
      Let $R$ be a commutative Noetherian ring with identity such that every finitely generated projective $R$-module is free. Let $F:P \to \RMod$ be a nonzero p.f.g. zigzag persistence module with extrema $\{z_1,z_2, \ldots, z_n\}$ for $n \geq 2$ satisfying the PCC. Then $F$ is interval decomposable.
   \end{Corollary}

   \begin{proof}
      Immediate from Theorem~\ref{thm:fin-zigzag-pc-decomp} and Lemma~\ref{lem:free-interval-decomp}. 
   \end{proof}

\section{Zigzag Persistence Modules: Infinite Extrema}\label{sec:infinite-zigzag}

   Throughout this section, we assume that $R$ is a commutative Noetherian ring with identity, unless otherwise specified. Several of the arguments in this section are inspired by Botnan's approach to decomposing zigzag persistence modules over fields indexed by infinite discrete zigzag posets~\cite{bot17}. As in Section~\ref{sec:total-order} on totally ordered persistence modules, we first prove in Proposition~\ref{prp:inf-single-pc-decomp} that a pointwise finitely generated zigzag persistence module with infinite extrema satisfying the PCC admits a decomposition containing a single direct summand of projective constant type. We then use this result to prove Theorems~\ref{thm:inf-complete-pc-decomp} and \ref{thm:inf-zigzag-pc-decomp-nat}.

   %%% Notation: Brackets with Interval Notation %%%
   \begin{Notation}
      Let $P$ be a zigzag poset. Recall the four interval notations $[x,y]$, $[x,y)$, $(x,y]$ and $(x,y)$ outlined in Notation~\ref{not:zigzag-interval}. We use bracket notation to indicate the following:

      \begin{enumerate}
         \item $\langle x,y]$ is an interval of the form $(x,y]$ or $[x,y]$;
         \item $[x,y \rangle$ is an interval of the form $[x,y)$ or $[x,y]$;
         \item $\langle x,y)$ is an interval of the form $(x,y)$ or $[x,y)$;
         \item $(x,y \rangle$ is an interval of the form $(x,y)$ or $(x,y]$;
         \item $\langle x,y \rangle$ is any one of the four intervals $[x,y]$, $[x,y)$, $(x,y]$ or $(x,y)$.
      \end{enumerate}

      Additionally, we will denote a persistence module of projective constant type over the interval $\langle x,y \rangle$ by $\mathbb{I}^{\langle x,y \rangle}$, with notation for other intervals defined analogously.
   \end{Notation}

   %%% Lemma: Extending a Finite Interval Decomposition %%%
   \begin{Lemma}\label{lem:extend-single-pc-decomp}
      Let $F:P \to \RMod$ be a p.f.g. zigzag persistence module with extrema $\{z_i \mid i \in \Z\}$ satisfying the PCC. Let $\sqsubseteq$ be the associated total order. If there exists an interval $[z_s,z_t]$ such that the projective constant decomposition of $F|_{[z_s,z_t]}$ includes a submodule supported on $\langle r,w \rangle$ with $z_s \sqsubset r \sqsubset w \sqsubset z_t$, then $F$ admits a decomposition $F = \mathbb{I}^{\langle r,w \rangle} \oplus G$ for some submodule $\mathbb{I}^{\langle r,w\rangle}$ of projective constant type. 
   \end{Lemma}

   \begin{proof}
      Note that $F|_{[z_s,z_t]}$ admits a projective constant decomposition by Theorem~\ref{thm:fin-zigzag-pc-decomp}. Suppose ${J}$ is a direct summand in the projective constant decomposition of $F|_{[z_s,z_t]}$ that is supported on $\langle r,w \rangle$ with $z_s \sqsubset r \sqsubset w \sqsubset z_t$. Then we can write $F|_{[z_s,z_t]} = {J} \oplus H$ where $H$ is the direct sum of the remaining direct summands in the decomposition.
      
      Define the persistence submodules $\mathbb{I}^{\langle r,w \rangle}$ and $G$ as follows:
      \begin{align*}
         \mathbb{I}^{\langle r,w \rangle}_x = 
         \begin{cases} 
            {J}_x & \text{if } z_s \sqsubseteq x \sqsubseteq z_t, \\
            0 & \text{otherwise;} \\
         \end{cases} 
         \hspace{15pt} 
         & G_x = 
         \begin{cases} 
            H_x & \text{if } z_s \sqsubseteq x \sqsubseteq z_t, \\
            F_x &  \text{otherwise.} \\
         \end{cases}
      \end{align*}

      This results in a decomposition $F=\mathbb{I}^{\langle r,w \rangle} \oplus G$ as desired, where $\mathbb{I}^{\langle r,w \rangle}$ is a persistence submodule of projective constant type supported on $\langle r,w \rangle$.
   \end{proof}

   For a totally ordered poset with distinct minimum and maximum elements, a totally ordered persistence module indexed by this poset can be treated as a zigzag persistence module with two finite extrema consisting of the minimum and maximum, and whose associated total order coincides with the total ordering of $P$. The following lemmas characterize the conditions under which such a persistence module decomposes into an internal direct sum of a single module of projective constant type and a complementary submodule.

   %%% Lemma: Injective Decomposition on [r,w] %%%
   \begin{Lemma}\label{lem:injective-decomp}
      Let $P$ be a totally ordered poset with distinct minimum and maximum elements; that is, there exist $r,w \in P$ with $r \neq w$ such that $r \leq x$ for all $x \in P$ and $w \geq x$ for all $x \in P$. Suppose that $F:P \to \RMod$ is a totally ordered persistence module such that:
      \begin{enumerate}
         \item[(a)] $\Coker(F_{r,w})$ is projective;
         \item[(b)] every morphism $F_{x,y}$ is injective;
         \item[(c)] $F_r=M \oplus N$ for some nonzero projective $R$-submodule $M$.
      \end{enumerate}

      Then there exists a decomposition $F=\mathbb{I}^{[r,w]} \oplus G$, where:
      \begin{enumerate}
         \item $\mathbb{I}^{[r,w]}$ is a persistence module of projective constant type supported on $[r,w]$;
         \item $\mathbb{I}^{[r,w]}_x \cong M$ for all $r \leq x \leq w$;
         \item $\mathbb{I}^{[r,w]}_{r}=M$ and $G_{r}=N$.
      \end{enumerate}
   \end{Lemma}

   \begin{proof}
      Define $\mathbb{I}^{[r,w]}_x = F_{r,x}(M)$ for all $r \leq x \leq w$. Observe that $\mathbb{I}^{[r,w]}_r=M$. Additionally, define $N_{w} = F_{r,w}(N)$.

      By injectivity, we have $\Ima[r,w] = \mathbb{I}^{[r,w]}_{w} \oplus N_{w}$. Since $\Coker[r,w]$ is projective by assumption, there exists a submodule $H$ such that $F_{w}=\Ima[r,w] \oplus H$.
      
      Define $G_{w} = N_{w} \oplus H$. Then $F_{w}=\mathbb{I}^{[r,w]}_{w}\oplus G_{w}$. Additionally, define $G_x=\Inv{F_{x,w}}(G_{w})$ for all $r \leq x<w$. Observe that $G_r=N$. 
      We claim $F_x=\mathbb{I}^{[r,w]}_x \oplus G_x$ for all $x \in P$. Indeed, let $a \in F_x$. Then $F_{x,w}(a)=b+c$ for some $b \in \mathbb{I}^{[r,w]}_{w}, c \in G_{w}$. Moreover, there exists $d \in \mathbb{I}^{[r,w]}_{r}$ such that $F_{r,w}(d)=b$. We have the following implications:
      \begin{flalign*}
         & F_{x,w}(a)=b+c = F_{x,w}F_{r,x}(d)+c\\
         \implies & F_{x,w}(a - F_{r,x}(d)) = c \\
         \implies & a - F_{r,x}(d) \in G_x \\
         \implies & a \in \mathbb{I}^{[r,w]}_x + G_x.
      \end{flalign*}

      To show the sum is direct, let $a \in \mathbb{I}^{[r,w]}_x \cap G_x$. Then $F_{x,w}(a) \in \mathbb{I}^{[r,w]}_{w} \cap G_{w}$. Hence $F_{x,w}(a)=0$ and injectivity implies $a=0$. This proves the claim, and moreover by construction we have $F_{x,y}(\mathbb{I}^{[r,w]}_x)=\mathbb{I}^{[r,w]}_y$ and $F_{x,y}(G_x) \subseteq G_y$ for all $x \leq y$. Hence $F=\mathbb{I}^{[r,w]} \oplus G$ is the desired decomposition. 
   \end{proof}

   %%% Lemma: Surjective Decomposition on [r,w] %%%
   \begin{Lemma}\label{lem:surjective-decomp}
      Let $P$ be a totally ordered poset with distinct minimum and maximum elements; that is, there exist $r,w \in P$ with $r \neq w$ such that $r \leq x$ for all $x \in P$ and $w \geq x$ for all $x \in P$. Suppose that $F:P \to \RMod$ is a totally ordered persistence module such that:
      \begin{enumerate}
         \item[(a)] every morphism $F_{x,y}$ is surjective;
         \item[(b)] $F_w$ is a projective $R$-module;
         \item[(c)] $F_w=M \oplus N$ for some nonzero projective $R$-submodule $M$.
      \end{enumerate}

      Then there exists a decomposition $F=\mathbb{I}^{[r,w]} \oplus G$, where:
      \begin{enumerate}
         \item $\mathbb{I}^{[r,w]}$ is a persistence module of projective constant type supported on $[r,w]$;
         \item $\mathbb{I}^{[r,w]}_x \cong M$ for all $r \leq x \leq w$;
         \item $\mathbb{I}^{[r,w]}_{w}=M$ and $G_{w}=N$.
      \end{enumerate}
   \end{Lemma}

   \begin{proof}
      Since $F_{r,w}$ is surjective with $F_{w}=M \oplus N$ projective, by Lemma~\ref{lem:decomp-surj-map} we can choose a decomposition $F_{r}=\mathbb{I}^{[r,w]}_{r} \oplus G_{r}$ such that:
      \begin{enumerate}
         \item[(i)] $\mathbb{I}^{[r,w]}_r \cong M$;
         \item[(ii)] $F_{r,w}(\mathbb{I}^{[r,w]}_r) = M$;
         \item[(iii)] $F_{r,w}(G_r) = N$. 
      \end{enumerate}

      Define $\mathbb{I}^{[r,w]}_x=F_{r,x}(\mathbb{I}^{[r,w]}_r)$ and $G_x=F_{r,x}(G_r)$ for all $r < x \leq w$. Observe that $\mathbb{I}^{[r,w]}_w=M$ and $G_w=N$. We claim $F_x=\mathbb{I}^{[r,w]}_x \oplus G_x$ for all $x \in P$.
      
      Indeed, clearly $\mathbb{I}^{[r,w]}_x + G_x$ generates $F_x$ by surjectivity of $F_{r,x}$. To see the sum is direct, suppose $a \in \mathbb{I}^{[r,w]}_x \cap G_x$. Then there exists $b \in \mathbb{I}^{[r,w]}_{r}$ with $F_{r,x}(b)=a$. In particular, we have

      \begin{center}
         $F_{r,w}(b)=F_{x,w}F_{r,x}(b)=F_{x,w}(a) =0$
      \end{center}

      since $F_{x,w}(a) \in M \cap N$. But $F_{r,w}$ maps $\mathbb{I}^{[r,w]}_r$ to $\mathbb{I}^{[r,w]}_w$ bijectively, hence $F_{r,w}(b)=0 \implies b=0 \implies a= 0$.

      This proves the claim, and moreover by construction we have $F_{x,y}(\mathbb{I}^{[r,w]}_x)=\mathbb{I}^{[r,w]}_y$ and $F_{x,y}(G_x) \subseteq G_y$ for all $x \leq y$. Therefore $F=\mathbb{I}^{[r,w]} \oplus G$ is the desired decomposition. 
   \end{proof}

   %%% Proposition: Decomposing a Single Interval %%%
   \begin{Proposition}\label{prp:inf-single-pc-decomp}
      Let $F:P \to \RMod$ be a nonzero p.f.g. zigzag persistence module with extrema $\{z_i \mid i \in \Z\}$  satisfying the PCC. Then there exists an internal direct sum decomposition $F = \mathbb{I} \oplus G$, where $\mathbb{I}$ is a module of projective constant type. 
   \end{Proposition}

   \begin{proof}
      %%% Set Up %%%
      \begingroup
      Let $\sqsubseteq$ denote the associated total order of the zigzag poset $P$. Without loss of generality, assume that for all $x,y \in P$ such that $z_i\sqsubseteq x \sqsubseteq y \sqsubseteq z_{i+1}$:
      \begin{align*}
         \begin{cases} 
            i \mbox{ is odd } \quad \Longrightarrow \quad x\leq y; \\
            i \mbox{ is even } \quad \Longrightarrow \quad  y\leq x.
         \end{cases}  
         \end{align*}

      See Definition~\ref{def:zigzag-poset} for a reminder on the above statement.
      
      Recall by Theorem~\ref{thm:fin-zigzag-pc-decomp} that $F|_{[z_{-n},{z_n}]}$ admits a projective constant decomposition for any positive integer $n$. We split this proof into two cases. 

      \textbf{Case 1.} Suppose there exists a positive integer $n$ such that the projective constant decomposition of $F|_{[z_{-n},{z_n}]}$ contains a submodule supported on $\langle r,w \rangle$ with $-n \sqsubset r \sqsubset w \sqsubset n$. Then we are done by Lemma~\ref{lem:extend-single-pc-decomp}. 

      \textbf{Case 2.} Suppose no such $n$ exists. Fix an index $z_i$ such that $F_{z_i} \neq 0$. Note that such an index exists since if all $F_{z_i}=0$, then we would have Case 1.
      \endgroup

      %%% Claim 1: Existence of Sufficient Injective/Surjective %%%
      \textbf{Claim~\ref*{prp:inf-single-pc-decomp}.1.} There exists $z_t \sqsupset z_i$ such that $F_{x,y}$ is injective and $F_{y,x}$ is surjective for all $z_t \sqsubseteq x \sqsubseteq y$, when such maps are defined. Dually, there exists $z_s \sqsubset z_i$ such that $F_{y,x}$ is injective and $F_{x,y}$ is surjective for all $x \sqsubseteq y \sqsubseteq z_s$, when such maps are defined. 
      \begin{proof}[Proof of Claim~\ref*{prp:inf-single-pc-decomp}.1.] 	
         To see this, we will show there are finitely many morphisms $F_{x,y}$ that are non-injective for $z_i \sqsubset x \sqsubseteq y$. Suppose by way of contradiction there is no $z_t \sqsupset z_i$ such that all maps $F_{x,y}$ are injective for $z_t \sqsubseteq x \sqsubseteq y$. Then there exist indices $x_1, y_1 \sqsupset z_i$ with $x_1 \leq y_1$ such that $\Ker[x_1,y_1] \neq 0$. 

         Recall from the Definition~\ref{def:zigzag-poset}(2) there exists $k_1 \in \Z$ such that $z_{k_1-1} \sqsubseteq x_1 \sqsubseteq y_1 \sqsubseteq z_{k_1}$. In particular, we have $x_1 \leq y_1 \leq z_{k_1}$ with respect to the partial order. Since $\Ker[x_1, y_1] \neq 0$, it follows that $\Ker[x_1,z_{k_1}] \neq 0$.
         
         Looking at the projective constant decomposition of $F|_{[z_i, z_{k_1}]}$ which exists by Theorem~\ref{thm:fin-zigzag-pc-decomp}, there exists a submodule $\mathbb{I}^{\langle r_1,w_1 \rangle}$ of projective constant type in the decomposition with $z_i \sqsubseteq r_1 \sqsubset w_1 \sqsubset z_{k_1}$. Now if $z_i \sqsubset r_1 \sqsubset w_1 \sqsubset z_{k_1}$, then we would have Case 1. Therefore $r_1=z_i$ and the submodule of projective constant type $\mathbb{I}^{[z_i,w_1 \rangle }$ is supported on the interval $[z_i,w_1 \rangle$. 

         Continuing in this manner, there exists an $x_2 \sqsupset z_{k_1}$ and $z_{k_2} \geq x_2$ such that $\Ker[x_2,z_{k_2}] \neq 0$. The projective constant decomposition of $F|_{[z_i,z_{k_2}]}$ will contain a submodule $\mathbb{I}^{[z_i,w_2 \rangle}$ of projective constant type supported on the interval $[z_i,w_2\rangle$ where $w_1 \sqsubset z_{k_1} \sqsubset w_2$. Repeating this process indefinitely results in an infinite ascending chain of $R$-submodules 
         
         \begin{center}
            $\mathbb{I}_{z_i}^{[z_i,w_1\rangle} \subseteq \mathbb{I}_{z_i}^{[z_i,w_1\rangle} \oplus \mathbb{I}_{z_i}^{[z_i,w_2\rangle} \subseteq \mathbb{I}_{z_i}^{[z_i,w_1\rangle} \oplus \mathbb{I}_{z_i}^{[z_i,w_2\rangle} \oplus \mathbb{I}_{z_i}^{[z_i,w_3\rangle} \subseteq \cdots$
         \end{center}

         of $F_{z_i}$, a contradiction since $F_{z_i}$ is Noetherian. Therefore such a sufficiently large index $z_t$ must exist. The settings for surjections and for $z_s \sqsubset z_i$ are analogous.
      \end{proof}

      %%% Beginning Decomposition %%%
      \begingroup
      Let $z_s$ and $z_t$ be the indices as described in Claim~\ref*{prp:inf-single-pc-decomp}.1. Without loss of generality, we can choose $s$ and $t$ to be odd integers.

      By Theorem~\ref{thm:fin-zigzag-pc-decomp}, $F|_{[z_s,z_t]}$ admits a projective constant decomposition. Choose a submodule $\mathbb{I}^{\langle r,w \rangle}$ in this decomposition supported on $\langle r,w \rangle$ so that we write $F|_{[z_s,z_t]} = \mathbb{I}^{\langle r,w \rangle} \oplus H$ where $H$ is the direct sum of the remaining submodules in the decomposition. Our goal is to extend this decomposition to all of $F$ so that we can write $F=\mathbb{I} \oplus G$. 
      
      Define $G_x = H_x$ and $\mathbb{I}_x=\mathbb{I}^{\langle r,w \rangle}_x$ for all $z_s \sqsubseteq x \sqsubseteq z_t$. It remains to carefully define what happens for indices $x \sqsupset z_t$ and $x \sqsubset z_s$.  
      \endgroup

      %%% Infinite Decomposition %%%
      \begingroup
      First we will focus on indices $x \sqsupset z_t$. If $\mathbb{I}^{\langle r,w \rangle}_{z_t}=0$ and $H_{z_t}=F_{z_t}$, then the decomposition is easy. Define $G_x = F_x$ and $\mathbb{I}_x=0$ for all $x \sqsupset z_t$.

      If $\mathbb{I}^{\langle r,w \rangle}_{z_t} \neq 0$, define $\mathbb{I}_{z_t}=\mathbb{I}^{\langle r,w \rangle}_{z_t}$ and $G_{z_t}=H_{z_t}$. Since $t$ is odd, we have $z_t < z_{t+1}$ with respect to the partial order. In particular, $F|_{[z_t,z_{t+1}]}$ is a totally ordered persistence module indexed by the poset $[z_t,z_{t+1}]$ with a minimum $z_t$ and maximum $z_{t+1}$ such that:
      \begin{enumerate}
         \item[(a)] $\Coker(F_{z_t,z_{t+1}})$ is projective;
         \item[(b)] every morphism $F_{x,y}$ is injective for $z_t \leq x \leq y \leq z_{t+1}$;
         \item[(c)] $F_{z_t}=\mathbb{I}_{z_t} \oplus G_{z_t}$, where $\mathbb{I}_{z_t}$ is a nonzero projective $R$-module.
      \end{enumerate}
      
      By Lemma~\ref{lem:injective-decomp}, there exists a decomposition $F|_{[z_t,z_{t+1}]} = \mathbb{I}^{[z_t,z_{t+1}]} \oplus H^{(t)}$ such that:
      \begin{enumerate}
         \item $\mathbb{I}^{[z_t,z_{t+1}]}$ is a persistence submodule of projective constant type supported on $[z_t,z_{t+1}]$;
         \item $\mathbb{I}^{[z_t,z_{t+1}]}_x \cong \mathbb{I}_{z_t}$ for all $z_t \sqsubseteq x \sqsubseteq z_{t+1}$;
         \item $\mathbb{I}^{[z_t,z_{t+1}]}_{z_t} = \mathbb{I}_{z_t}$ and $H^{(t)}_{z_t}=G_{z_t}$.
      \end{enumerate}

      Define $\mathbb{I}_x=\mathbb{I}^{[z_t,z_{t+1}]}_x$ and $G_x = H^{(t)}_x$ for all $z_t \sqsubset x \sqsubseteq z_{t+1}$. This, in effect, has extended the decomposition over $[z_s,z_t]$ to $[z_s,z_{t+1}]$. 

      Similarly, since $t+1$ is even, we have $z_{t+2} < z_{t+1}$ with respect to the partial order. In particular, $F|_{[z_{t+1},z_{t+2}]}$ is a totally ordered persistence module indexed by the poset $[z_{t+1},z_{t+2}]$ with a minimum $z_{t+2}$ and maximum $z_{t+1}$ such that:
      \begin{enumerate}
         \item[(a)] every morphism $F_{y,x}$ is surjective for $z_{t+1} \sqsubseteq x \sqsubseteq y \sqsubseteq z_{t+2}$;
         \item[(b)] $F_{z_{t+1}}$ is a projective $R$-module;
         \item[(c)] $F_{z_{t+1}}=\mathbb{I}_{z_{t+1}} \oplus G_{z_{t+1}}$, where $\mathbb{I}_{z_{t+1}}$ is a nonzero projective $R$-module.
      \end{enumerate}
      
      By Lemma~\ref{lem:surjective-decomp}, there exists a decomposition $F|_{[z_{t+1},z_{t+2}]} = \mathbb{I}^{[z_{t+1},z_{t+2}]} \oplus H^{(t+1)}$ such that:

      \begin{enumerate}
         \item $\mathbb{I}^{[z_{t+1},z_{t+2}]}$ is a persistence submodule of projective constant type supported on $[z_{t+1},z_{t+2}]$;
         \item $\mathbb{I}^{[z_{t+1},z_{t+2}]}_x \cong \mathbb{I}_{z_{t+1}}$ for all $z_{t+1} \sqsubseteq x \sqsubseteq z_{t+2}$;
         \item $\mathbb{I}^{[z_{t+1},z_{t+2}]}_{z_{t+1}} = \mathbb{I}_{z_{t+1}}$ and $H^{(t+1)}_{z_{t+1}}=G_{z_{t+1}}$.
      \end{enumerate} 

      Define $\mathbb{I}_x=\mathbb{I}^{[z_{t+1},z_{t+2}]}_x$ and $G_x = H^{(t+1)}_x$ for all $z_{t+1} \sqsubset x \sqsubseteq z_{t+2}$. This, in effect, has extended the decomposition over $[z_s,z_{t+1}]$ to $[z_s,z_{t+2}]$.
      
      We can repeat this process indefinitely over the intervals $[z_{t+2}, z_{t+3}], [z_{t+3}, z_{t+4}]$, and so on to obtain a decomposition $F_x = \mathbb{I}_x \oplus G_x$ for all $x \sqsupset z_t$.
      
      Dually, we can apply the above argument for indices $x \sqsubset z_s$. This results in a decomposition $F = \mathbb{I} \oplus G$ as desired.
      \endgroup
   \end{proof}

   We now proceed to the main results of this section. Namely, the following two theorems are the statement of Theorem~\ref{thm:main-pc-zigzag} when $\Gamma=\Z$ and $\Gamma=\N$.

   %%% Theorem: Complete Projective Constant Decomposition in Z %%%
   \begin{Theorem}\label{thm:inf-complete-pc-decomp}
      Let $F:P \to \RMod$ be a nonzero p.f.g. zigzag persistence module with extrema $\{z_i \mid i \in \Z\}$ satisfying the PCC. Then $F$ admits a projective constant decomposition. 
   \end{Theorem}

   \begin{proof}
      Immediate from Lemma~\ref{lem:decomp-criteria}, Proposition~\ref{prp:inf-single-pc-decomp} and Lemma~\ref{lem:pcc-more-facts} when letting $\mathcal{E}$ in Lemma~\ref{lem:decomp-criteria} be the set of submodules of $F$ of projective constant type. 
   \end{proof}

   %%% Theorem: Projective Constant Decomposition for Extrema in N %%%
   \begin{Theorem}\label{thm:inf-zigzag-pc-decomp-nat}
      Let $F:P \to \RMod$ be a nonzero p.f.g. zigzag persistence module with extrema $\{z_i \mid i \in \N\}$ satisfying the PCC. Then $F$ admits a projective constant decomposition. 
   \end{Theorem}

   \begin{proof}
      There is an obvious way to embed a zigzag persistence module $F$ with extrema $\{z_i \mid i \in \N\}$ into a zigzag persistence module $\bar{F}$ with extrema $\{z_i \mid i \in \Z\}$ so that $\bar{F}|_{\{x \sqsupseteq z_0\}}=F$ and $\bar{F}|_{\{x \sqsubset z_0\}} =0$. The projective constant decomposition of $\bar{F}$ guaranteed by Theorem~\ref{thm:inf-complete-pc-decomp} yields a projective constant decomposition of $F$.
   \end{proof}

   We deduce Corollary~\ref{cor:main-pid-zigzag} when $\Gamma=\Z$ and $\Gamma=\N$ in the following result.

   %%% Corollary: PID or Local Ring %%%
   \begin{Corollary}\label{cor:inf-pid-zigzag-pc-decomp}
      Let $R$ be a commutative Noetherian ring with identity such that every finitely generated projective $R$-module is free. Let $F:P \to \RMod$ be a nonzero p.f.g. zigzag persistence module with either extrema $\{z_i \mid i \in \Z\}$ or extrema $\{z_i \mid i \in \N\}$ satisfying the PCC. Then $F$ is interval decomposable.
   \end{Corollary}

   \begin{proof}
      Immediate from Theorem~\ref{thm:inf-complete-pc-decomp}, Theorem~\ref{thm:inf-zigzag-pc-decomp-nat} and Lemma~\ref{lem:free-interval-decomp}. 
   \end{proof}

\section{Necessity}\label{sec:necessity}

   Throughout this section, we assume that $R$ is a (not necessarily commutative) Noetherian ring with identity, unless otherwise specified. In the preceding sections, we showed that the projective colimit conditions are sufficient to guarantee the existence of projective constant decompositions of $A_n$-persistence modules and zigzag persistence modules. We further established that requiring the cokernels of all internal morphisms to be projective is sufficient to guarantee the existence of a projective constant decompositions of totally ordered persistence modules. These conditions also implied the existence of interval decompositions when $R$ had the added property that all finitely generated projective $R$-modules are free. We now show that these conditions are in fact necessary whenever a projective constant (respectively, interval) decomposition exists. 

   %%% Lemma: Necessity for Zigzag Finite Extrema PC Module %%%
   \begin{Lemma}\label{lem:nec-zigzag-interval-pc-finite}
      Let $F:P \to \RMod$ be a nonzero zigzag persistence module with finite extrema $\{z_1,z_2, \cdots, z_n\}$ for $n \geq 2$ of projective constant type. Then:
      \begin{enumerate}
         \item[\textbf{(N1)}] $\Colim(F)$ is projective;
         \item[\textbf{(N2)}] $\Coker(F_{z_1} \to \Colim(F))$ is projective;
         \item[\textbf{(N3)}] $\Coker(F_{z_n} \to \Colim(F))$ is projective;
         \item[\textbf{(N4)}] $\Coker(F_{z_1} \oplus F_{z_n} \to \Colim(F))$ is projective.
      \end{enumerate}
   \end{Lemma}

   \begin{proof}
      %%% Set Up %%%
      \begingroup
      Let $\sqsubseteq$ denote the associated total order. By Definition~\ref{def:pc-module}, $F$ is isomorphic to $F(I, M)$ for some interval $I$ in P and for some projective $R$-module $M$. It suffices to prove the conditions hold for $F=F(I, M)$. Recall this means that $F_x=M$ for all $x \in I$ and $F_x=0$ otherwise. Additionally, $F_{x,y}=\Id_{x,y}$ for all $x \leq y$ in $I$ and $F_{x,y}=0$ otherwise. 

      When $n=2$, either $\Colim(F)=F_{z_2}$ if $z_1 < z_2$ or $\Colim(F)=F_{z_1}$ if $z_2 < z_1$. This observation makes the four conditions immediate.
      
      Let $n \geq 3$. We divide the remainder of this proof into three cases based on whether $z_1$ and $z_n$ are sinks or sources (see Definitions~\ref{def:source-and-sink} and \ref{def:zigzag-poset}). In each case, we show that the four conditions hold.
      \endgroup

      %%% Case 1: Two Sinks %%%
      \begingroup
      \textbf{Case 1.} Suppose that $z_1$ and $z_n$ are both sinks. This implies that $n$ is odd, and that for all $x,y \in P$ such that $z_i \sqsubseteq x \sqsubseteq y \sqsubseteq z_{i+1}$:
      \begin{align*}
         \begin{cases} 
            i \mbox{ is odd } \quad \Longrightarrow \quad y \leq x; \\
            i \mbox{ is even } \quad \Longrightarrow \quad  x \leq y.
         \end{cases}  
         \end{align*}
      
      Below is an example diagram of $F$; the thick lines indicate the interval $I$ where $F$ takes the value $M$, and the arrows indicate the direction of the morphisms.
      
      
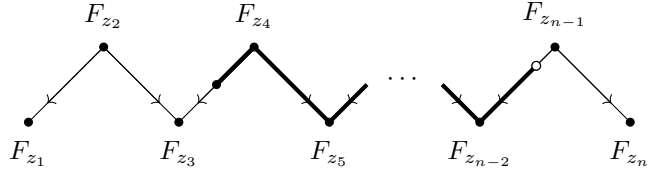
\begin{center}
         \begin{tikzpicture}
            %%% Main Lines %%%
            \draw[black, -] (-5,-0.5) -- (-4,0.5);
            \draw[black, -] (-3,-0.5) -- (-4,0.5);
            \draw[black, -] (-3,-0.5) -- (-2,0.5);
            \draw[black, ultra thick, -] (-2,0.5) -- (-1,-0.5);
            \draw[black, -] (3,-0.5) -- (2,0.5);
            \draw[black, -] (2,0.5) -- (1,-0.5);
            
            %%% Partial Lines %%%
            \draw[black, ultra thick, -] (-2.5,0) -- (-2,0.5);
            \draw[black, ultra thick, -] (1,-0.5) -- (1.75,0.25);
            \draw[black, ultra thick, -] (-2.5,0) -- (-2,0.5);
            \draw[black, ultra thick, -] (-1,-0.5) -- (-0.5,0);
            \draw[black, ultra thick, -] (1,-0.5) -- (0.5,0);

            \node[label={90:$\cdots$}] at (0,-0.25) {};
            
            %%% Arrows %%%
            \draw[black, ->] (-4,0.5) -- (-4.75,-0.25);
            \draw[black, ->] (-4,0.5) -- (-3.25,-0.25);
            \draw[black, ->] (-2,0.5) -- (-2.75,-0.25);
            \draw[black, ->] (2,0.5) -- (2.75,-0.25);
            \draw[black, ->] (-2,0.5) -- (-1.25,-0.25);
            \draw[black, ->] (2,0.5) -- (1.25,-0.25);
            \draw[black, ->] (-0.5,0) -- (-0.75,-0.25);
            \draw[black, ->] (0.5,0) -- (0.75,-0.25);

            %%% Points %%%	
            \filldraw [black] (-5,-0.5) circle (1.5pt);
            \node[label={270:$F_{z_1}$}] at (-5,-0.5) {};
            
            \filldraw [black] (-4,0.5) circle (1.5pt);
            \node[label={90:$F_{z_2}$}] at (-4,0.5) {};

            \filldraw [black] (-3,-0.5) circle (1.5pt);
            \node[label={270:$F_{z_3}$}] at (-3,-0.5) {};
            
            \filldraw [black] (-2,0.5) circle (1.5pt);
            \node[label={90:$F_{z_4}$}] at (-2,0.5) {};

            \filldraw [black] (-1,-0.5) circle (1.5pt);
            \node[label={270:$F_{z_5}$}] at (-1,-0.5) {};

            \filldraw [black] (1,-0.5) circle (1.5pt);
            \node[label={270:$F_{z_{n-2}}$}] at (1,-0.5) {};
            
            \filldraw [black] (2,0.5) circle (1.5pt);
            \node[label={90:$F_{z_{n-1}}$}] at (2,0.5) {};

            \filldraw [black] (3,-0.5) circle (1.5pt);
            \node[label={270:$F_{z_n}$}] at (3,-0.5) {};

            \filldraw [black] (-2.5,0) circle (1.5pt);
            \node at (-2.5,0) {};

            \filldraw [black] (1.75,0.25) circle (1.7pt);
            \node at (1.75,0.25) {};

            \filldraw [white] (1.75,0.25) circle (1.2pt);
            \node at (1.75,0.25) {};
         \end{tikzpicture}
      \end{center}

      Let $K=\{1,3, \ldots, n-2, n\}$ be the odd integers of the sinks $z_i$ of $F$, and let $C=\{2, 4, \ldots, n-3, n-1\}$ be the even integers of the sources $z_i$ of $F$. It follows that 
      
      \begin{center}
         $\Colim(F)= \Frac{\bigoplus_{i \in K} F_{z_i}}{E} = \Frac{F_{z_1} \oplus F_{z_3} \oplus \cdots \oplus F_{z_{n-2}} \oplus F_{z_n}}{E}$.
      \end{center}
      
      The submodule $E$ above is generated by the set of elements
      
      \begin{center}
         $\{\iota_{z_{i-1}}(F_{z_i,z_{i-1}}(a)) - \iota_{z_{i+1}}(F_{z_i,z_{i+1}}(a)) \mid a \in F_{z_i}, i \in C\}$,
      \end{center}

      where $\iota_{z_j}:F_{z_j} \to \bigoplus_{i \in K} F_{z_i}$ is the natural inclusion map for all sinks $z_j$ with $j \in K$. This defines an equivalence relation $\sim$ on $\bigoplus_{i \in K} F_{z_i}$ that identifies 
      
      \begin{center}
         $\iota_{z_{i-1}}(F_{z_i,z_{i-1}}(a)) \sim \iota_{z_{i+1}}(F_{z_i,z_{i+1}}(a))$
      \end{center} 

      for all sources $z_i$ with $i \in C$ and for all $a \in F_{z_i}$.  
      
      Consider the interval $I$. If no extrema $z_i$ lie in $I$, then $\Colim(F)=0$ since $F_{z_i}=0$ for all sinks $z_i$ with $i \in K$. The four conditions are immediate. 
      
      Now suppose that at least one extremum $z_i$ lies in $I$. Let $z_s$ and $z_t$ be the minimum and maximum extrema in $I$, respectively, with respect to the associated total order. Note that it is possible for $z_s=z_t$ when exactly one extremum lies in $I$.

      %%% Claim 1: Interval Starts with Sources Implies Colimit is 0 %%%
      \textbf{Claim~\ref*{lem:nec-zigzag-interval-pc-finite}.1.} If $z_s$ or $z_t$ are sources, then $\Colim(F)=0$.

      \begin{proof}[Proof of Claim~\ref*{lem:nec-zigzag-interval-pc-finite}.1.]

      Without loss of generality, assume that $z_s$ is a source. Hence $F_{z_s}=M$. Since $z_s$ is the minimum extrema in $I$, this implies $F_{z_{s-1}}=0$. Based on the equivalence relation $\sim$, we identify
      
      \begin{center}
         $0 = \iota_{z_{s-1}}(F_{z_s,z_{s-1}}(a)) \sim \iota_{z_{s+1}}(F_{z_s,z_{s+1}}(a))$ for all $a \in F_{z_s}$.
      \end{center}
      
      If $z_s$ is the maximum source in $P$, then $z_{s+1}=z_n$. Hence all elements of $\bigoplus_{i \in K} F_{z_i}$ get identified with $0$ under the equivalence relation $\sim$, and thus $\Colim(F)=0$.

      If $z_{s}$ is not the maximum source in $P$, then $z_{s+2}$ is a source in $P$. Based on the equivalence relation $\sim$, we identify

      \begin{center}
         $\iota_{z_{s+1}}(F_{z_{s+2},z_{s+1}}(a)) \sim \iota_{z_{s+3}}(F_{z_{s+2},z_{s+3}}(a))$ for all $a \in F_{z_{s+2}}$. 
      \end{center}

      But this implies
      \begin{center}
         $0 \sim \iota_{z_{s+1}}(F_{z_s,z_{s+1}}(a)) \sim \iota_{z_{s+1}}(F_{z_{s+2},z_{s+1}}(a))$ for all $a \in F_{z_{s+2}} \subseteq F_{z_s}$.
      \end{center}

      Continuing this process, it is easy to see that all elements of $\bigoplus_{i \in K} F_{z_i}$ get identified with $0$ under the equivalence relation $\sim$, and hence $\Colim(F)=0$.
      \end{proof}

      Based on the setting of Claim~\ref*{lem:nec-zigzag-interval-pc-finite}.1 when $z_s$ or $z_t$ are sources, the four conditions are immediate. Finally, assume that $z_s$ and $z_t$ are both sinks. Then we can equivalently describe the submodule $E$ as
      
      \begin{center}	
         $E = \Set{(a_{z_1}, a_{z_{3}}, \ldots, a_{z_{n-2}}, a_{z_{n}})}{\sum_{i \in K}a_{z_i}=0}$.
      \end{center}

      We omit the detailed calculations of this as it is easy to verify using the key observation that if $z_i$ is a source in $I$, then $z_{i-1}$ and $z_{i+1}$ are also in $I$. 

      Let $\theta: \bigoplus_{i \in K} F_{z_i} \to M$ be the map given by 
      
      \begin{center}
         $(a_{z_1}, a_{z_{3}}, \ldots, a_{z_{n-2}}, a_{z_{n}}) \mapsto \sum_{i \in K}a_{z_i}$. 
      \end{center}

      Then $\Ker(\theta) = E$, and moreover $\theta$ is surjective. Thus 
      
      \begin{center}
         $\Colim(F) = \Frac{\bigoplus_{i \in K} F_{z_i}}{E} = \Frac{\bigoplus_{i \in K} F_{z_i}}{\Ker(\theta)} \cong \Ima(\theta)= M$ 
      \end{center}

      by the First Isomorphism Theorem, which proves condition \textbf{(N1)} that $\Colim(F)$ is projective.	
      
      If $F_{z_1}=0$ then \textbf{(N2)} is immediate. If $F_{z_1}=M$, then $z_1=z_s$. The map $F_{z_1} \to \Colim(F)$ is easily seen to be surjective. Hence $\Coker(F_{z_1} \to \Colim(F)) = 0$ is trivially projective, which shows \textbf{(N2)}. Analogous arguments show that conditions \textbf{(N3)} and \textbf{(N4)} also hold.
      \endgroup

      %%% Case 2: Two Sources %%%
      \begingroup
      \textbf{Case 2.} Suppose that $z_1$ and $z_n$ are both sources. This implies that $n$ is odd, and that for all $x,y \in P$ such that $z_i \sqsubseteq x \sqsubseteq y \sqsubseteq z_{i+1}$:
      \begin{align*}
         \begin{cases} 
            i \mbox{ is odd } \quad \Longrightarrow \quad x \leq y; \\
            i \mbox{ is even } \quad \Longrightarrow \quad  y \leq x.
         \end{cases}  
         \end{align*}
      
      If $n=3$, then $\Colim(F)=F_{z_2}$ where $F_{z_2}=0$ or $F_{z_2}=M$. Hence $\Colim(F)$ is projective satisfying \textbf{(N1)}, and the other three conditions follow trivially. 
      
      If $n>3$, then $\Colim(F)=\Colim(F|_{[z_2, z_{n-1}]})$ where $z_2$ and $z_{n-1}$ are both sinks. Hence we apply Case 1 to show that $\Colim(F)$ is projective for condition \textbf{(N1)}. 
      
      If $\Colim(F|_{[z_2, z_{n-1}]})=0$, then the other three conditions are immediate. Suppose that $\Colim(F|_{[z_2, z_{n-1}]}) \cong M$. Then Claim~\ref*{lem:nec-zigzag-interval-pc-finite}.1 implies that the minimum and maximum extrema $z_s$ and $z_t$, respectively, in $I \cap [z_2,z_{n-1}]$ are both sinks.

      If $F_{z_1}=0$, then \textbf{(N2)} is immediate. If $F_{z_1} = M$, then $z_1 \in I$. Since $z_1 \sqsubseteq z_2 \sqsubseteq z_s$ with $z_1,z_s \in I$, Lemma~\ref{lem:interval-in-zigzag} implies $z_2 \in I$, and hence $z_2=z_s$ by minimality of $z_s$. It is easy to verify from this observation that $F_{z_1} \to \Colim(F)$ is a surjective map. Hence $\Coker(F_{z_1} \to \Colim(F))=0$ is trivially projective, which shows \textbf{(N2)} is satisfied. Conditions \textbf{(N3)} and \textbf{(N4)} follow similarly. 
      \endgroup

      %%% Case 3: One Sink and One Source %%%
      \begingroup
      \textbf{Case 3.} Suppose that $z_1$ is a sink and $z_n$ is a source. This implies that $n$ is even, and that for all $x,y \in P$ such that $z_i \sqsubseteq x \sqsubseteq y \sqsubseteq z_{i+1}$:
      \begin{align*}
         \begin{cases} 
            i \mbox{ is odd } \quad \Longrightarrow \quad y \leq x; \\
            i \mbox{ is even } \quad \Longrightarrow \quad  x \leq y.
         \end{cases}  
         \end{align*}
      
      Observe that $\Colim(F)=\Colim(F|_{[z_1,z_{n-1}]})$ where $z_1$ and $z_{n-1}$ are both sinks. Hence we apply Case 1 to show that $\Colim(F)$ is projective for condition \textbf{(N1)}.

      If $\Colim(F|_{[z_1,z_{n-1}]})=0$, then the other three conditions are immediate. Suppose that $\Colim(F|_{[z_1,z_{n-1}]}) \cong M$. Then Claim~\ref*{lem:nec-zigzag-interval-pc-finite}.1 implies that the minimum and maximum extrema $z_s$ and $z_t$, respectively, in $I \cap [z_1,z_{n-1}]$ are both sinks.
      
      If $F_{z_1}=0$, then \textbf{(N2)} is immediate. If $F_{z_1} = M$, then $z_1 \in I$. Hence $z_1=z_s$ by minimality of $z_s$. It is easy to verify that $F_{z_1} \to \Colim(F)$ is a surjective map. Hence $\Coker(F_{z_1} \to \Colim(F))=0$ is trivially projective, which shows \textbf{(N2)}.
      
      If $F_{z_n}=0$, then \textbf{(N3)} is immediate. If $F_{z_n}=M$, then $z_n \in I$. Since $z_t \sqsubseteq z_{n-1} \sqsubseteq z_n$ with $z_t,z_n \in I$, Lemma~\ref{lem:interval-in-zigzag} implies $z_{n-1} \in I$, and hence $z_{n-1}=z_t$ by maximality of $z_t$. It is easy to verify with this observation that $F_{z_n} \to \Colim(F)$ is a surjective map. Hence $\Coker(F_{z_n} \to \Colim(F))=0$ is trivially projective, which shows \textbf{(N3)}.

      Condition \textbf{(N4)} follows similarly as \textbf{(N2)} and \textbf{(N3)}. The case for when $z_1$ is a source and $z_n$ is a sink is analogous.
      \endgroup
   \end{proof}

   %%% Lemma: Necessity for Zigzag Infinite Extrema PC Module %%%
   \begin{Lemma}\label{lem:nec-zigzag-interval-pc}
      Let $F:P \to \RMod$ be a nonzero zigzag persistence module with extrema $\{z_i \mid i \in \Gamma \}$ of projective constant type. Let $\sqsubseteq$ be the associated total order. Then $F$ satisfies the projective colimit conditions; that is, for any $x \sqsubseteq y$ of $P$:
      \begin{enumerate}
         \item[\textbf{(C1)}] $\Colim(F|_{[x,y]})$ is projective;
         \item[\textbf{(C2)}] $\Coker(F_x \to \Colim(F|_{[x,y]}))$ is projective;
         \item[\textbf{(C3)}] $\Coker(F_y \to \Colim(F|_{[x,y]}))$ is projective;
         \item[\textbf{(C4)}] $\Coker(F_x \oplus F_y \to \Colim(F|_{[x,y]}))$ is projective.
      \end{enumerate}
   \end{Lemma}

   \begin{proof}
      By Definition~\ref{def:pc-module}, $F$ is isomorphic to $F(I, M)$ for some interval $I$ in P and for some projective $R$-module $M$. It suffices to prove the conditions hold for $F=F(I, M)$.  
         
      Let $x \sqsubseteq y$ be indices in $P$. When $x=y$, the four conditions are immediate. Similarly, when $[x,y] \cap I = \emptyset$, the four conditions are immediate. Let $x \sqsubset y$ such that $[x,y] \cap I \neq \emptyset$. 

      There exist $s,t \in \Gamma$ such that $z_{s} \sqsubseteq x \sqsubseteq z_{s+1}$ and $z_t \sqsubseteq y \sqsubseteq z_{t+1}$ by Lemma~\ref{lem:zigzag-shape}. Observe that $F|_{[x,y]}$ is a nonzero zigzag persistence module of projective constant type with either:
      \begin{enumerate}
         \item finite extrema $\{x,y\}$ if $s=t$;
         \item finite extrema $\{x, z_{s+1}, z_{s+2}, \ldots, z_{t-1},z_{t}, y \}$ if $s < t$.
      \end{enumerate}
      In either case, Lemma~\ref{lem:nec-zigzag-interval-pc-finite} guarantees the four conditions are satisfied. 
   \end{proof}

   The following result is the statement of Theorem~\ref{thm:main-pc-zigzag-nec} regarding zigzag persistence modules.

   %%% Theorem: Zigzag Necessity %%%
   \begin{Theorem}\label{thm:nec-zigzag-pc}
      Let $F:P \to \RMod$ be a nonzero p.f.g. zigzag persistence module with extrema $\{z_i \mid i \in \Gamma \}$. If $F$ admits a projective constant (respectively, interval) decomposition, then $F$ satisfies the projective colimit conditions.
   \end{Theorem}

   \begin{proof}
      Let $F = \bigoplus_{G \in \mathcal{A}}G$ be a projective constant decomposition of $F$. Then each $G$ satisfies the PCC by Lemma~\ref{lem:nec-zigzag-interval-pc}. It is clear that $F$ will also satisfy the PCC based on the fact that colimits commute with coproducts and cokernels in the category $\RMod$ \cite{par70}*{\S2, Corollary 2}. In particular, it is essential that $F$ is pointwise finitely generated so that only finitely many $G_x$ are nontrivial for any fixed $x \in P$. An interval decomposition of $F$ is a special case of a projective constant decomposition, which makes the respective claim immediate.
   \end{proof}

   Consequently, we also establish Theorem~\ref{thm:main-pc-quiver-nec} regarding $A_n$-persistence modules.

   %%% Theorem: A_n Quiver Necessity %%%
   \begin{Theorem}\label{thm:nec-quiver-pc}
      Let $F:A_n \to \RMod$ be a nonzero p.f.g. $A_n$-persistence module. If $F$ admits a projective constant (respectively, interval) decomposition, then $F$ satisfies the projective colimit conditions.
   \end{Theorem}

   \begin{proof}
      Suppose that $F$ admits a projective constant decomposition. When $n=1$, the result is immediate. When $n \geq 2$, $F$ can be viewed as a nonzero p.f.g. zigzag persistence module with finite extrema which satisfies the PCC by Theorem~\ref{thm:nec-zigzag-pc}. An interval decomposition of $F$ is a special case of a projective constant decomposition, which makes the respective claim immediate.
   \end{proof}

   Finally, we state Theorem~\ref{thm:main-pc-total-nec} regarding totally ordered persistence modules.

   %%% Theorem: Totally Ordered Necessity %%%
   \begin{Theorem}\label{thm:nec-total-pc}
      Let $F:P \to \RMod$ be a nonzero p.f.g. totally ordered persistence module. If $F$ admits a projective constant (respectively, interval) decomposition, then the cokernel of every internal morphism $F_{x,y}$ is projective.
   \end{Theorem}

   \begin{proof}
      When $x=y$, the claim is immediate. When $x<y$, $F|_{[x,y]}$ can be viewed as a nonzero p.f.g. zigzag persistence module with finite extrema $\{x,y\}$. Since $F$ admits a projective constant decomposition, $F|_{[x,y]}$ also admits a projective constant decomposition and therefore $F|_{[x,y]}$ satisfies the PCC by Theorem~\ref{thm:nec-zigzag-pc}. In particular, $\Coker(F_{x,y})$ is projective by Lemma~\ref{lem:pcc-more-facts}. An interval decomposition of $F$ is a special case of a projective constant decomposition, which makes the respective claim immediate.
   \end{proof}

   Note in Theorems~\ref{thm:nec-zigzag-pc},~\ref{thm:nec-quiver-pc}, and \ref{thm:nec-total-pc}, the necessity of the conditions hold regardless of whether $R$ is commutative or noncommutative. This is in contrast to the sufficiency results, which in some cases required $R$ to be commutative.

\section{Uniqueness}\label{sec:uniqueness}

   We conclude this article by investigating the uniqueness of projective constant decompositions. Given that modules over Noetherian rings need not admit a unique decomposition into indecomposable summands, uniqueness of projective constant decompositions cannot be expected in full generality. 

   Nevertheless, by imposing additional conditions on $\RMod$, we obtain meaningful uniqueness results. We first demonstrate that if $\RMod$ is a Krull-Schmidt category—an additive category where objects decompose uniquely into finite direct sums of indecomposables with local endomorphism rings—then every projective constant decomposition of a persistence module admits a refinement to an indecomposable version (see Definition~\ref{def:indec-pc-decomp}). This decomposition is unique up to isomorphism and permutation of summands. Furthermore, we show that when $R$ is an integral domain, an interval decomposition of any pointwise finitely generated persistence module indexed by a small poset category likewise satisfies this uniqueness property. Notably, we relax the Noetherian requirement on $R$ throughout this section, assuming only that $R$ is a unital ring and that $\RMod$ denotes the category of all unitary left $R$-modules. Additionally, the results are established for persistence modules indexed by arbitrary posets, extending beyond the cases of $A_n$-type quivers, totally ordered sets, or zigzag posets presented in this article.

   %%% Defintion: Indecomposable Persistence Module %%%
   \begin{Definition}
      Let $P$ be a poset. We say a persistence module $F:P \to \RMod$ is \defn{indecomposable} if $F$ is nonzero and $F \cong G \oplus H$ implies either $G=0$ or $H=0$. Likewise, an $R$-module $M$ is \defn{indecomposable} if $M$ is nonzero and $M \cong N \oplus L$ implies either $N=0$ or $L=0$.
   \end{Definition}

   %%% Lemma: F(I,M) is Indecomposable %%%
   \begin{Lemma}\label{lem:pc-constant-decomposable}
      Let $P$ be a poset and let $F:P \to \RMod$ be a module of projective constant type with $F \cong F(I,M)$ for some interval $I$ in $P$ and for some indecomposable projective $R$-module $M$. Then $F$ is indecomposable.
   \end{Lemma}

   \begin{proof}
      It suffices to prove the result for $F(I,M)$. Suppose that $F(I,M)$ admits a decomposition $F(I,M) \cong G \oplus H$. Then for each $x \in I$, we have $F(I,M)_x \cong G_x \oplus H_x$. Since $F(I,M)_x = M$ is an indecomposable projective $R$-module, it follows that either $G_x=0$ for all $x \in I$ or $H_x=0$ for all $x \in I$, implying that $F(I,M)$ is indecomposable.
   \end{proof}

   %%% Defintion: Indecomposable Projective Constant Decomposition %%%
   \begin{Definition}\label{def:indec-pc-decomp}
      Let $P$ be a poset and let $F:P \to \RMod$ be a persistence module. An \defn{indecomposable projective constant decomposition} of $F$ is a projective constant decomposition $F = \bigoplus_{G \in \mathcal{A}} G$ such that for every $G \in \mathcal{A}$, we have $G \cong F(I,M)$ for some interval $I$ in $P$ and for some indecomposable projective $R$-module $M$.
   \end{Definition}

   %%% Lemma: F(I,M) Local Endomorphism Ring %%%
   \begin{Lemma}\label{lem:krull-local-endo}
      Let $P$ be a poset and let $F:P \to \RMod$ be a module of projective constant type with $F \cong F(I,M)$ for some interval $I$ in $P$ and for some indecomposable projective $R$-module $M$. If $\RMod$ is a Krull-Schmidt category, then $F$ has a local endomorphism ring.
   \end{Lemma}

   \begin{proof}
      It suffices to prove the result for $F(I,M)$. Let $f: F(I,M) \to F(I,M)$ be a morphism. Observe that for all $r,w \in I$ such that $r \leq w$, we have 
      \[ F(I,M)_{r,w} \circ f_r  = f_w \circ F(I,M)_{r,w}.\]
      Hence $f_r = f_w$ since $F(I,M)_{r,w} = \Id_M$.

      Now consider any $x,y \in I$. Since $I$ is connected, there exist $y_0, y_1, \ldots, y_n \in I$ such that Definition \ref{def:interval}(2) holds. By the observation above, we will have $f_x=f_{y_0}=f_{y_1}=\cdots=f_{y_n}=f_y$. This yields an isomorphism 
      \[ \End(F(I,M)) \cong \End_R(M),\]
      where $\End(F(I,M))$ and $\End_R(M)$ are the endomorphism rings of $F(I,M)$ and $M$, respectively. As $M$ is indecomposable, $\End_R(M)$ is local, which in turn ensures that $\End(F(I, M))$ is local.
   \end{proof}

   We now state and prove Theorem~\ref{thm:main-krull-uniqueness}. 

   %%% Theorem: Uniqueness of Indecomposable Projective Constant Decompositions %%%
   \begin{Theorem}\label{thm:krull-uniqueness}
      Let $P$ be a poset and let $F:P \to \RMod$ be a nonzero persistence module. Suppose that $\RMod$ is a Krull-Schmidt category. If $F$ admits a projective constant decomposition, then $F$ admits a unique indecomposable projective constant decomposition up to isomorphism and permutation of the summands.
   \end{Theorem}

   \begin{proof}
      Let $F = \bigoplus_{G \in \mathcal{A}} G$ be a projective constant decomposition of $F$. For each $G \in \mathcal{A}$, we have $G \cong F(I,M)$ for some interval $I$ in $P$ and for some projective $R$-module $M$. Since $\RMod$ is a Krull-Schmidt category, we can decompose $M$ into a finite direct sum of indecomposable projective $R$-modules. This induces a finite decomposition $G = \bigoplus_{i=1}^n G^{(i)}$ where each $G^{(i)} \cong F(I,N_i)$ for some indecomposable projective $R$-module $N_i$. Thus each $G \in \mathcal{A}$ admits an indecomposable projective constant decomposition, and hence $F$ admits an indecomposable projective constant decomposition.

      To see uniqueness, let $F = \bigoplus_{H \in \mathcal{B}} H$ be an indecomposable projective constant decomposition of $F$. By Lemmas~\ref{lem:pc-constant-decomposable} and \ref{lem:krull-local-endo}, each $H \in \mathcal{B}$ is indecomposable with local endomorphism ring. The Krull-Remak-Schmidt-Azumaya theorem~\cite{azu50} guarantees this decomposition is unique up to isomorphism and permutation of the summands.
   \end{proof}

   %%% Corollary: Uniqueness of Interval Decompositions %%%
   \begin{Corollary}\label{cor:uniquness-int}
      Let $P$ be a poset and let $F:P \to \RMod$ be a nonzero persistence module. Suppose that $\RMod$ is a Krull-Schmidt category and that $R$ is an indecomposable $R$-module. If $F$ admits an interval decomposition, then the interval decomposition is unique up to isomorphism and permutation of the summands.
   \end{Corollary}

   \begin{proof}
      Immediate from Theorem~\ref{thm:krull-uniqueness}, Lemma~\ref{lem:pc-constant-decomposable} and Lemma~\ref{lem:krull-local-endo}.
   \end{proof}

   The assumption in Corollary~\ref{cor:uniquness-int} that $R$ is indecomposable as an $R$-module is essential. This condition ensures that an interval decomposition is indeed an indecomposable projective constant decomposition.

   Recall that a category is \defn{small} if both its collection of objects and its collection of morphisms form sets. The poset categories of $A_n$-type quivers, totally ordered sets, and zigzag posets are all small. For persistence modules indexed by small poset categories, we establish the uniqueness of interval decompositions over integral domains without assuming the Krull-Schmidt property. This is achieved via scalar extension to the field of fractions, which reduces the problem to the known uniqueness results for persistence modules over fields.

   %%% Definition: Tensor Product of Persistence Modules %%%
   \begin{Definition}
      Let $R$ be an integral domain, $\mathfrak{F}$ the field of fractions of $R$, and $\Vect(\mathfrak{F})$ the category of $\mathfrak{F}$-vector spaces. Let $P$ be a poset and let $F:P \to \RMod$ be a persistence module. Define $\mathfrak{F} \otimes F: P \to \Vect(\mathfrak{F})$ to be the persistence module such that:
      \begin{itemize}
         \item $(\mathfrak{F} \otimes F)_x = \mathfrak{F} \otimes_R F_x$ for all $x \in P$;
         \item $(\mathfrak{F} \otimes F)_{x,y} = \Id_{\mathfrak{F}} \otimes F_{x,y}$ for all $x \leq y$ in $P$. 
      \end{itemize}
   \end{Definition}

   We now state and prove Theorem~\ref{thm:main-domain-uniqueness}. This result is obtained by generalizing the proof of \cite{luo23}*{\S4, Lemma 1}, which establishes the uniqueness of interval decompositions for pointwise freely and finitely generated persistence modules over PIDs indexed by finite linear quivers.

   %%% Theorem: Uniqueness of Interval Decompositions over Integral Domains %%%
   \begin{Theorem}\label{thm:domain-uniqueness}
      Let $R$ be an integral domain. Let $P$ be a small poset category and let $F:P \to \RMod$ be a nonzero p.f.g. persistence module. If $F$ admits an interval decomposition, then this decomposition is unique up to isomorphism and permutation of the summands.
   \end{Theorem}

   \begin{proof}
      Let $\mathfrak{F}$ be the field of fractions of $R$ and $\Vect(\mathfrak{F})$ the category of $\mathfrak{F}$-vector spaces. Let $F=\bigoplus_{G \in \mathcal{A}} G$ be an interval decomposition of $F$. For each $x \in P$, we have
      \begin{flalign*}
         (\mathfrak{F} \otimes F)_x 
         = \mathfrak{F} \otimes_R F_x 
         = \mathfrak{F} \otimes_R \bigg( \bigoplus_{G \in \mathcal{A}} G_x \bigg) 
         \cong \bigoplus_{G \in \mathcal{A}} (\mathfrak{F} \otimes_R G_x)
         \cong \bigoplus_{G \in \mathcal{A}} (\mathfrak{F} \otimes G)_x,
      \end{flalign*}
      where the first isomorphism follows from the fact that only finitely many $G_x$ are nontrivial. Hence there is an obvious way to define a persistence module isomorphism such that
      \[ \mathfrak{F} \otimes F \cong \bigoplus_{G \in \mathcal{A}} (\mathfrak{F} \otimes G).\] 
      In particular, if $G$ is supported on the interval $I$ in $P$, then $\mathfrak{F} \otimes G$ is an interval module supported on $I$. Indeed, for all $x \in I$ we have
      \[ (\mathfrak{F} \otimes G)_x = \mathfrak{F} \otimes_R G_x \cong \mathfrak{F} \otimes_R R \cong \mathfrak{F}, \]
      and moreover for all $x,y \in I$ such that $x \leq y$, we see that the morphism
      \[ (\mathfrak{F} \otimes G)_{x,y} = \Id_{\mathfrak{F}} \otimes G_{x,y}\]
      is a surjection between one-dimensional $\mathfrak{F}$-vector spaces, and hence an isomorphism. Thus $\mathfrak{F} \otimes F$ is interval decomposable, and this decomposition is unique up to isomorphism and permutation of the summands whenever $P$ is a small poset category \cite{bot18}*{\S2}. The obvious bijection between the interval summands of $F$ and those of $\mathfrak{F} \otimes F$ guarantees the interval decomposition of $F$ is likewise unique up to isomorphism and permutation of the summands. 
   \end{proof}

   %%% Corollary: Uniqueness of Interval Decompositions over Specific Rings %%%
   \begin{Corollary}\label{cor:domain-unique-examples}
      Let $R$ be one of the following rings:
      \begin{enumerate}
         \item a principal ideal domain;
         \item the polynomial ring in finite variables $k[x_1, \ldots, x_n]$, where $k$ is a principal ideal domain or a field;
         \item a local Noetherian integral domain;
         \item the ring of formal power series $S[[x]]$, where is $S$ is a local Noetherian integral domain.
      \end{enumerate}
      Let $P$ be a poset and let $F:P \to \RMod$ be a nonzero p.f.g. persistence module. If $F$ admits an interval decomposition, then this decomposition is unique up to isomorphism and permutation of the summands.
   \end{Corollary}

   \begin{proof}
      Immediate from Theorem~\ref{thm:domain-uniqueness} noting every ring listed is an integral domain.
   \end{proof}

   Recall that the rings listed in Corollary~\ref{cor:domain-unique-examples} have the property that every finitely generated projective module is free (see Corollary~\ref{cor:free-interval-decomp-examples}). Consequently, Corollaries~\ref{cor:main-pid-quiver}, \ref{cor:main-pid-total}, \ref{cor:main-pid-zigzag}, and \ref{cor:domain-unique-examples} collectively establish the existence and essential uniqueness of interval decompositions for persistence modules over such rings indexed by $A_n$-type quivers, totally ordered sets, and zigzag posets, provided the relevant necessary and sufficient conditions are met.

\end{document}